\documentclass[11pt]{article}

\usepackage[english]{babel}
\usepackage[latin1]{inputenc}
\usepackage[T1]{fontenc}
\usepackage{amsmath,amssymb,amsthm}
\usepackage{latexsym}
\usepackage{bbm}
\usepackage[all]{xy}
\usepackage[bookmarks=false]{hyperref}

\DeclareMathOperator{\coker}{coker}

\DeclareMathOperator{\diam}{diam}
\DeclareMathOperator{\ext}{ext}
\DeclareMathOperator{\Ext}{Ext}

\DeclareMathOperator{\HH}{H}

\DeclareMathOperator{\Hom}{Hom}

\DeclareMathOperator{\PGor}{PGor}
\DeclareMathOperator{\Proj}{Proj}

\DeclareMathOperator{\Spec}{Spec}

\DeclareMathOperator{\VV}{V}

\newtheorem{theorem}{Theorem}[section]
\newtheorem{lemma}[theorem]{Lemma}
\newtheorem{corollary}[theorem]{Corollary}
\newtheorem{definition}[theorem]{Definition}
\newtheorem{example}[theorem]{Example}
\newtheorem{proposition}[theorem]{Proposition}

\newtheorem{remark}[theorem]{Remark}

\newcommand{\pp}{{\mathbb P}}

\newcommand{\sE}{{\mathcal E}}
\newcommand{\sF}{{\mathcal F}}

\newcommand{\sH}{{\mathcal H}}
\newcommand{\sI}{{\mathcal I}}

\newcommand{\sN}{{\mathcal N}}
\newcommand{\sO}{{\mathcal O}}

\newcommand{\mfm}{\mathfrak m}

\newcommand{\proj}[1]{\mathbbm P^{#1}}

\newcommand{\eext}[3]{\sideset{_{#1}}{_{#2}^{#3}}\ext}
\newcommand{\EExt}[3]{\sideset{_{#1}}{_{#2}^{#3}}\Ext}
\newcommand{\HHom}[2]{\sideset{_{#1}}{_{#2}}\Hom}

\newcommand{\mmid}%
{\makebox[0pt][l]{$\shortmid$}\raisebox{1.7pt}[0pt][0pt]{$\shortmid$}}
\newcommand{\inup}{{\cup\mspace{-8.2mu}\mmid\mspace{8.2mu}}}

\begin{document}

\title{Liaison invariants and the Hilbert scheme of codimension 2 subschemes in $\proj{n+2}$}

\author{Jan O. Kleppe}


\date{\ }

\maketitle

\vspace*{-0.55in}
\begin{abstract}
  \noindent In this paper we study the Hilbert scheme ${\rm
    Hilb}^{p(v)}(\proj{})$ of equidimensional locally Cohen-Macaulay
  codimension 2 subschemes, with a special look to surfaces in $\proj{4}$ and
  3-folds in $\proj{5}$, and the Hilbert scheme stratification
  $\HH_{\gamma,\rho}$ of constant cohomology. For every $(X) \in {\rm
    Hilb}^{p(v)}(\proj{})$ we define a number $\delta_X$ in terms of the
  graded Betti numbers of the homogeneous ideal of $X$ and we prove that $1+
  \delta_X - \dim_{(X)}\HH_{\gamma,\rho}$ and $1+ \delta_X - \dim
  T_{\gamma,\rho}$ are CI-biliaison invariants where $T_{\gamma,\rho}$ is the
  tangent space of $\HH_{\gamma,\rho}$ at $(X)$. As a corollary we get a
  formula for the dimension of any generically smooth component of ${\rm
    Hilb}^{p(v)}(\proj{})$ in terms of $ \delta_X $ and the CI-biliaison
  invariant. Both invariants are equal in this case.

  Recall that, for space curves $C$, Martin-Deschamps and Perrin have proved
  the smoothness of the ``morphism'' $\phi : \HH_{\gamma,\rho} \rightarrow
  E_{\rho}$ := isomorphism classes of graded modules $M$ satisfying $\dim M_v
  = \rho(v)$, given by sending $C$ onto its Rao module. For surfaces $X$ in
  $\proj{4}$ we have two Rao modules $M_i \simeq \oplus H^i(\sI_X(v))$ of
  dimension $\rho_i(v)$, $\rho:= ( \rho_1, \rho_2 )$ and an induced extension
  $b \in \EExt {0}{}{2} (M_2,M_1)$ and a result of Horrocks and Rao saying
  that a triple $D := (M_1,M_2,b)$ of modules $M_i$ of finite length and an
  extension $b$ as above determine a surface $X$ up to biliaison. We prove
  that the corresponding ``morphism'' $\varphi: \HH_{\gamma,\rho} \to \VV_\rho
  = $ isomorphism classes of graded modules $M_i$ satisfying $\dim (M_i)_v =
  \rho_i(v)$ and commuting with $b$, is smooth, and we get a smoothness
  criterion for $ \HH_{\gamma,\rho}$, i.e. for the equality of the two
  biliaison invariants. Moreover we get some smoothness results for ${\rm
    Hilb}^{p(v)}(\proj{})$, valid also for 3-folds, and we give examples of
  obstructed surfaces and 3-folds. The linkage result we prove in this paper
  turns out to be useful in determining the structure and dimension of
  $\HH_{\gamma,\rho}$, and for proving the main biliaison theorem above.



\end{abstract}
 \vspace*{-0.15in}
{\small
 \tableofcontents}
\thispagestyle{empty}

\section{Introduction.}

A main object of this paper is to find the dimension of the Hilbert scheme,
${\rm Hilb}^{p(v)}(\proj{})$, of equidimensional locally Cohen-Macaulay (lCM)
codimension 2 subschemes of $\proj{}:= \proj{n+2}$. As an initial ambitious
goal we look for a formula for the dimension of any reduced component $V$ of
the Hilbert scheme ${\rm Hilb}^{p(v)}(\proj{})$ in terms of the graded Betti
numbers of the homogeneous ideal $I_X$ of a general element $(X)$ of $V$.
Somehow we expect the matrices in the minimal resolution of $I_X$ to play a
role, but it seems that only the cohomology groups of $\sO_X$ contribute since
we succeed in reaching our goal up to a biliaison invariant! Indeed in this
paper we explicitly define an invariant $ \delta_X^{n+1}(-n-3)$ in terms of
the graded Betti numbers of $I_X$ and $ H_*^n(\sO_X)$ and we prove that
$$\dim V= 1+ \delta_X^{n+1}(-n-3) - {\rm sumext}(X)$$ where $ {\rm sumext}(X)$
is a CI-biliaison invariant (Corollary~\ref{thmeulerICMhigher}). In the case
$X$ is a curve $(n=1)$ with Hartshorne-Rao module $M$, we have
\begin{equation*} 
\ {\rm sumext}(X) = \sum_{i=0}^1 \eext 0Ri (M,M)  \ , \ 
\end{equation*}
and there is a similar, but much more complicated, formula in the surface case
(Remark~\ref{p3.4}). 

Let $\HH_{\gamma,\rho} \subseteq {\rm Hilb}^{p(v)}(\proj{})$ be the Hilbert
scheme whose $k$-points $(X)$ corresponds to equidimensional lCM codimension 2
subschemes $X$ of $\proj{n+2}$ with constant cohomology (see \cite{MDP1} for
the curve case). If $X$ is {\it any} equidimensional lCM codimension 2
subscheme of $\proj{}$, we define 
$ {\rm obsumext}(X)$ in the following way,
$$ {\rm obsumext}(X)= 1+ \delta_X^{n+1}(-n-3) - \dim_{(X)}
\HH_{\gamma,\rho}.$$ We define $ {\rm sumext}(X)$ by the same expression
provided we have replaced $\HH_{\gamma,\rho}$ by its tangent space,
$T_{\gamma,\rho}$, at $(X)$. Then we prove that $ {\rm sumext}(X)$ and $ {\rm
  obsumext}(X)$ are CI-biliaison invariants (Theorem~\ref{maint4.1}). Since
every arithmetically Cohen-Macaulay codimension 2 subscheme is in the liaison
class of a complete intersection (CI) by Gaeta's theorem, it follows that $
{\rm sumext}(X)= {\rm obsumext}(X) =0$ and that $ \dim_{(X)} {\rm
  Hilb}^{p(v)}(\proj{}) = 1+ \delta_X^{n+1}(-n-3)$ for $n>0$ if $X$ is
arithmetically Cohen-Macaulay (Corollary~\ref{themACM}).
Even though we do not prove the explicit
expression of $ {\rm sumext}(X)$ in terms the Rao modules of $X$ in general,
the theorem is motivated from the fact that the Rao modules are invariant under
biliaison up to shift. In fact it seems more effective to compute $ {\rm
  sumext}(X)$ and $ {\rm obsumext}(X)$ by considering a nice representative
$X'$ in its even liaison class, e.g. the minimal element, and to compute
$\delta_{X'}^{n+1}(-n-3),\ \dim_{(X')} \HH_{\gamma,\rho}$, and $ \dim
T_{\gamma,\rho}$ for $X'$.

Since the curve case of the results above is rather well understood
(\cite{MDP1}, \cite{krao}), we will in the present paper mostly concentrate on
the study of the Hilbert scheme $\HH(d,p,\pi)$ of surfaces of degree $d$ and
arithmetic (resp. sectional) genus $p$ (resp. $\pi$). Recall that, for space
curves $C$, 
Martin-Deschamps and Perrin 
proved the smoothness of the ``morphism'' $\phi : \HH_{\gamma,\rho}
\rightarrow E_{\rho}$: = isomorphism classes of graded $R$-modules $M$
satisfying $\dim M_v = \rho(v)$, given
by sending $C$ onto its Rao module. 
Earlier Rao proved that any graded $R$-module $M$ of finite length determines
the liaison class of a curve, up to dual and shift in the grading (\cite{R}).
Note that Rao's result is related to the surjectivity of $\phi$, while the
smoothness of $\phi$ implies infinitesimal surjectivity. For surfaces in
$\proj{4}$ there is a result in Bolondi's paper \cite{B2}, stating that a
triple $D := (M_1,M_2,b)$ of graded modules $M_i$ of finite length and an
extension $b \in \EExt {0}{}{2} (M_2,M_1)$ determine the biliaison class of a
surface $X$ such that $M_i \simeq \oplus H^i(\sI_X(v))$ modulo some shift in
the grading. The result is a consequence of the main theorem of \cite{R2} and
Horrocks' classification of stable vector bundles (\cite{Ho}), as mentioned by
Rao in \cite{R2}. Therefore it is natural to consider the stratification
$\HH_{\gamma,\rho}$ of $\HH(d,p,\pi)$ where now $\rho:= ( \rho_1, \rho_2 )$
and $\rho_i(v)=\dim H^i(\sI_X(v))$, and to ask for the smoothness of the
corresponding ``morphism'' $\varphi: \HH_{\gamma,\rho} \to \VV_\rho :=$
isomorphism classes of triples $(M_1,M_2,b)$ where $M_i$ are graded
$R$-modules which satisfy $\dim (M_i)_v = \rho_i(v)$ and where an isomorhpism
between triples is an isomorphism between the corresponding modules which
commutes with the extensions. We prove in section 5 that the answer is yes
(Theorem~\ref{t1.1}). 
As a corollary we get a smoothness criterion for $\HH_{\gamma,\rho}$
(Corollary~\ref{p1.3}, Remark~\ref{p3.4}), i.e. for the equality $ {\rm
  sumext}(X)= {\rm obsumext}(X)$ to hold. Note that since we do not prove that
the morphism $\varphi$ extends to a morphism of schemes, we only prove that
the corresponding morphism of the local deformation functors is formally
smooth. This, however, takes fully care of what we want.

In section 6 we determine the tangent space of $\HH_{\gamma,\rho}$ at $(X)$,
and we prove a local isomorphism $\HH_{\gamma,\rho} \simeq \HH(d,p,\pi)$ at
$(X)$ under some conditions (Proposition~\ref{dp3.1}, Remark~\ref{t3.7}).
Note, however, that if $X$ has seminatural cohomology, we know that
$\HH_{\gamma,\rho} \simeq \HH(d,p,\pi)$ at $(X)$ by the semicontinuity of $
\dim H^i(\sI_X(v))$ and this observation mostly suffices for our applications.
In section 7 we prove a useful linkage result (Theorem~\ref{t4.1}) which we
apply to determine the structure and the dimension of $\HH_{\gamma,\rho}$ and
to prove our main theorem on the biliaison invariants. In this section we also
give conditions for a linked surface to be e.g. non-generic, thus proving the
existence of surfaces with ``smaller'' cohomology in some cases
(Proposition~\ref{vanii}).

Since the technical problems in describing well the stratification of
$\HH(d,p,\pi)$ and the morphism $\phi$ are quite complicated (see \cite{K6}),
we don't follow up this trace for equidimensional lCM codimension 2 subschemes
$X \subseteq \proj{n+2}$ of dimension $n \ge 3$. Instead we only use our main
theorem on the biliaison invariance of ${\rm sumext}(X)$ and ${\rm
  obsumext}(X)$ together with some new results on the smoothness and the
dimension of ${\rm Hilb}^{p(v)}(\proj{})$ in our study of the Hilbert schemes
of e.g. 3-folds in section 9. We also give a vanishing criterion for
$h^1(\sN_X)$, but unfortunately, as in \cite{krao}, the results we get require
that the Hartshorne-Rao modules are rather ``small''. When the conditions of
these vanishing criteria do not hold, we give examples of obstructed surfaces
and 3-folds.

\textbf {Acknowledgment}. I heartily thank prof. G. Bolondi at Bologna for the
discussion with him on this topic. As the reader will see, especially for the
results in section 5 and 6, Bolondi's paper \cite{B2} is a main source of
ideas for the work presented here. It was prof. G. Bolondi who introduced me
to the idea of extending the results of \cite{B2}, as Martin-Deschamps and
Perrin do for space curves, to get a stratified description of the Hilbert
scheme $\HH(d,p,\pi)$, and who pointed out several interesting things to be
proved (see also \cite{K6}). Parts of the paper are also a natural
continuation of \cite{BM1} and \cite{BM2}. Moreover I warmly thank Hirokazu
Nasu at Chiba for his clarifying comments and useful Macaulay 2 computations
to the 
obstructed surface in Example~\ref{obstructed}, 
which led me to include examples of {\it smooth} obstructed
surfaces (Example~\ref{obstructedsm}).
%

\section{Notations and terminology.}
A surface (resp. curve) $X$ is an equidimensional, locally Cohen-Macaulay
subscheme (lCM) of $\proj{4}$ (resp. $ \proj{3}$) of dimension 2 (resp. 1)
with sheaf ideal $\sI_X$ and normal sheaf $\sN_X = \Hom_{\sO_{\proj{}}}
(\sI_X,\sO_X)$. If $\sF$ is a coherent $\sO_{\proj{}}$-Module, we let
$H^i(\sF) = H^i(\proj{},\sF)$, $H_*^i(\sF) = \oplus_v H^i(\sF(v))$ and
$h^i(\sF) = \dim H^i(\sF)$, and we denote by $\chi(\sF) = \Sigma (-1)^i
h^i(\sF)$ the Euler-Poincar\'{e} characteristic. Then $p(v)=\chi(\sO_X(v))$ is
the Hilbert polynomial of $X$. Put $n=\dim X$ and
\begin{equation*}
  \begin{aligned}
    s(X) & = \min \{v \arrowvert h^0(\sI_X(v)) \neq 0 \}, \\
    e(X) & = \max \{ v \arrowvert h^n(\sO_X(v)) \neq 0 \}.
  \end{aligned}
\end{equation*}
Let 
$I = I_X = H_*^0(\sI_X)$ be the homogeneous ideal. $I$ is a graded module over
the polynomial ring $R = k[X_0,X_1,..,X_{n+2}]$, where $k$ is supposed to be
algebraically closed (and of characteristic zero in section 5, 6 and say in
Example~\ref{obstructedsm} since we there use results and methods of papers
relying on this assumption). The postulation $\gamma$ of $X$ is the function
defined over the integers by $\gamma(v) = \gamma_X(v) = h^0 (\sI_X(v))$.

$X$ is \textit{unobstructed} if the Hilbert scheme ${\rm
  Hilb}^{p(v)}(\proj{n+2})$ (cf.\! \cite{G}) is smooth at the corresponding
point $(X)$, otherwise $X$ is
obstructed. 
A subscheme of $\proj{n+2}$ belonging to a sufficiently small open
irreducible subset of ${\rm Hilb}^{p(v)}(\proj{n+2})$ (small enough to satisfy
all the openness properties which we want it to have) is called a {\it
  generic} subscheme of ${\rm Hilb}^{p(v)}(\proj{n+2})$, and accordingly,
if we state that a generic subscheme has a certain property, then there is a
non-empty open irreducible 
subset of ${\rm Hilb}^{p(v)}(\proj{n+2})$ of subschemes having this property.

In the case of curves we put $\HH(d,g) = {\rm Hilb}^{p(v)}(\proj{n+2})$
provided $p(v) = dv +1-g$. Moreover we let $M = M(C):=H_*^1(\sI_C)$ be the
deficiency or Hartshorne-Rao module of the curve $C$. The deficiency function
$\rho$ is the defined by $\rho(v)= h^1 (\sI_C(v))$. Let
$\HH(d,g)_{\gamma,\rho}$ (resp. $\HH(d,g)_{\gamma}$) denote the subscheme of
$\HH(d,g)$ of curves with constant cohomology given by $\gamma$ and $\rho$,
(resp. constant postulation $\gamma$), see \cite{MDP1}. Let ${Def}_{M}$ be the
local deformation functor consisting of graded deformations $M_S$ of $M$ to $
\proj{3} \times \Spec (S)$ modulo graded isomorphisms of $M_S$ over $M$, where
$S$ is a local artinian $k$-algebra with residue field $k$, i.e. such that $M_S$
is $S$-flat and $M_S \otimes k = M$.

For a surface $X$ we define the arithmetic genus $p$ by $p = \chi (\sO_X)-1$,
while the sectional genus $\pi$ is given by $\chi (\sO_X(1)) = d-\pi +1+\chi
(\sO_X)$. By Riemann-Roch's theorem we have
\begin{equation} \label{hilbpolyX} p(v)= \chi (\sO_X(v)) = \frac{1}{2} dv^2
  -(\pi -1 -\frac 12 d)v + \chi (\sO_X).
\end{equation}
Put $\HH(d,p,\pi) = {\rm Hilb}^{p(v)}(\proj{n+2})$ in this case. Moreover let
$M_i = M_i(X)$ be the deficiency modules $H_*^i(\sI_X)$ for i = 1,2. The
deficiency $\rho = (\rho_1,\rho_2)$ of $X$ is the function defined over the
integers by $\rho(v) = \rho_X(v) = (\rho_1(v),\rho_2(v))$ where $ \rho_i(v) =
h^i (\sI_X(v))$ for $i = 1,2$. Let $\HH_{\gamma,\rho}=
\HH(d,p,\pi)_{\gamma,\rho}$ (resp. $\HH_{\gamma}=\HH(d,p,\pi)_{\gamma}$)
denote the subscheme of $\HH(d,p,\pi)$ of surfaces with constant cohomology
given by $\gamma$ and $\rho$, (resp. constant postulation
$\gamma$). 

For the notion of linkage, we refer to \cite{MIG}. Note that liaison (resp.
even liaison or biliaison) is the equivalence relation generated by linkage
(resp. direct linkages in an even number of steps).

For any graded $R$-module $N$, we have the right derived functors
$H_\mfm^i(N)$ and $\EExt v\mfm i (N,-)$ of $\Gamma_\mfm (N) = \oplus_v \ker
(N_v \to \Gamma (\proj{},\tilde N(v)))$ and $\Gamma_\mfm (\Hom_R (N,-))_v$
respectively (cf. \cite{SGA2}, exp. VI or \cite{H2}) where $\mfm =
(X_0,..,X_{n+2})$. We use small letters for the $k$-dimension and subscript
$v$ for the homogeneous part of degree $v$, e.g. $\eext v\mfm i (N_1,N_2) =
\dim \EExt v\mfm i (N_1,N_2)$.

Let $N_1$ and $N_2$ be graded $R$-modules of finite type. As in \cite{krao} we
need the spectral sequence
\begin{equation}\label{mspect}
  E_2^{p,q} = \EExt vRp (N_1, H_{\mfm}^q (N_2)) \Rightarrow
  \EExt v{\mfm}{p+q} (N_1, N_2)
\end{equation}
(\cite{SGA2}, exp. VI) and the duality isomorphism
\begin{equation}\label{mduality}
  \EExt v{\mfm}i (N_2, N_1) \cong \EExt {-v-n-3}R{n+3-i} (N_1, N_2)^\vee  , \ \
  i, v \in \mathbb Z
\end{equation}
where $(-)^\vee = \Hom_k( - , k)$ (cf. \cite{K5}, Thm. 1.1, see \cite{K2},
Thm. 2.1.4 for a full proof). Moreover there is a long exact sequence
\begin{equation}\label{seq}
  \to \EExt v{\mfm}i (N_1, N_2) \to \EExt vRi (N_1, N_2) \to
  \Ext_{\sO_\proj{}}^i (\tilde N_1, \tilde N_2 (v)) \to \EExt v{\mfm}{i+1}
  (N_1, N_2) \to
\end{equation}
(\cite{SGA2}, exp.\ VI) which at least for equidimensional, lCM subschemes of
codimension 2 (with $n > 0$) relate the deformation theory of $X$, described
by $H^{i-1} ({\sN_X}) \simeq \Ext_{\sO_\proj{}}^i (\tilde I, \tilde I )$ for
$i = 1,2$ (cf. \cite{K2}, Rem. 2.2.6), to the deformation theory of the
homogeneous ideal $I = I_X$, described by $\EExt 0Ri (I,I)$, in the following
exact sequence
\begin{equation}\label{mseq}
  \EExt vR1 (I,I) \hookrightarrow  H^0 (\sN_X (v)) \rightarrow
  \EExt v{\mfm}2 (I,I) \xrightarrow{\alpha} \EExt vR2 (I,I) \rightarrow 
  H^1 (\sN_X (v)) \rightarrow \EExt v{\mfm}3 (I,I) \rightarrow
\end{equation}
see \cite{W1} or \cite{F} for related works on such deformation functors.

\section{The dimension of $\HH(d,g)$ and biliaison invariants.}
In this section we consider the Hilbert scheme, $\HH(d,g)$, of curves in
$\proj{3}$ and results which we would like to generalize to surfaces in
$\proj{4}$. 
We will focus on the dimension of the Hilbert schemes and some biliaison
invariants which we naturally detect from this point of view.

Recall that $\chi(\sN_C(v)) = 2dv + 4d$ and that $\chi(\sN_C) = 4d$ is a lower
bound for $\dim_{(C)}\HH(d,g)$. For this reason the number $4d$ is often
called the expected dimension of $\HH(d,g)$ even though it often does not give
the correct dimension of $\HH(d,g)$ at $(C)$. E.g. for ACM curves the
dimension is always false if $e(C) \ge s(C)$.

To give a more reliable estimate for the dimension of the components of
$\HH(d,g)$, we have found it convenient to introduce the following invariant,
defined in terms of the numbers $n_{j,i}$ appearing in a minimal resolution of
the homogeneous ideal $I_C$ of $C$:
\begin{equation} \label{resoluC} 0 \to \bigoplus_{i=1}^{r_3} R(-n_{3,i}) \to
  \bigoplus_{i=1}^{r_2} R(-n_{2,i}) \to \bigoplus_{i=1}^{r_1} R(-n_{1,i}) \to
  I_C \to 0 \ .
\end{equation}
Note that we can define the graded Betti numbers, $\beta_{j,k}$, by just
putting $ \oplus_{k=1}^{\infty} R(-k)^{\beta_{j,k}}:= \oplus_{i=1}^{r_1}
R(-n_{j,i})$.
\begin{definition} \label{deltaC}
  If $C$ is a curve in $\proj{3}$, we let
\begin{equation*}
  \delta_C^{j}(v):= \sum_i h^j (\sI_C (n_{1,i}+v)) - \sum_i h^j (\sI_C
  (n_{2,i}+v))+ \sum_i h^j (\sI_C (n_{3,i}+v)). 
\end{equation*}
\end{definition}
Put $ \delta^{j}(v)= \delta_C^{j}(v)$. Then in \cite{krao} we proved the
following result (Lem.\! 2.2 of \cite{krao})

\begin{lemma}\label{eulerIC}
  Let $C$ be any curve of degree $d$ in $\proj{3}$. Then the following
  expressions are equal
  \begin{equation*} \label{lowerboundIC} {_0\!\ext_R^1 }(I_C,I_C )-
    {_0\!\ext_R^2 }(I_C,I_C ) = 1 - \delta^0(0) = 4d + \delta^2(0) -
    \delta^1(0) = 1 + \delta^2(-4) - \delta^1(-4).
\end{equation*}
\end{lemma}

\begin{remark} \label{remeulerIC} Comparing with the results and notations of
  \cite{MDP1} we recognize $1- \delta^{0}(0)$ as $\delta_{\gamma}$ and
  $\delta^{1}(-4)$ as $\epsilon_{\gamma ,\delta}$ in their terminology. By
  Lemma~\ref{eulerIC} it follows that the dimension of the Hilbert scheme
  $\HH_{\gamma,M}$ of constant postulation and Rao module, which they show is
  $\delta_{\gamma} + \epsilon_{\gamma,\delta} -{_0\!\hom}(M,M )$ (Thm. 3.8,
  page 171), is also equal to $1+\delta^{2}(-4)- {_0\!\hom}(M,M)$.
\end{remark}
Note that the difference of the $ {\ext }$-numbers in Lemma~\ref{eulerIC} is a
lower bound for $ \dim O_{\HH(d,g)_{\gamma},(C)}$ (\cite{krao}, proof of
Thm.\! 2.6 (i)). Mainly since $\HH(d,g)_{\gamma} $ is a subscheme of
$\HH(d,g)$, we used this lower bound in \cite{K08}, Thm.\! 24, to prove the
following result

\begin{theorem}\label{thmeulerIC}
  Let $C$ be a curve in $\proj{3}$ and let $ \delta^{j}(v)= \delta_C^{j}(v)$
  for any $j$ and $v$. Then the dimension of $\HH(d,g)$ at $(C)$ satisfies
\begin{equation*}
  \dim_{(C)}\HH(d,g) \ge  1 -
  \delta^0(0) = 4 d +  \delta^2(0) - \delta^1(0).
\end{equation*}
Moreover if $C$ is a generic curve of a generically smooth component $V$ of $
\HH(d,g)$ and $M= H_{*}^1(\sI_C)$, then
$$\dim V = 4 d + \delta^2(0) 
- \delta^1(0) +\ {_{-4}\!\hom_R}(I_C , M) $$ where ${_{-4}\!\Hom_R}(I_C , M)$
is the kernel of the map $$ \bigoplus_i H^1(\sI_C(n_{1,i}-4))
\rightarrow \bigoplus_i H^1(\sI_C(n_{2,i}-4)) $$ induced by the
corresponding map in \eqref{resoluC}.
\end{theorem}

\begin{remark} \label{thm26} Let $C$ be any curve in $\proj{3}$ and suppose $
  {_{-4}\!\Hom_R}(I_C,M) = {_{0}\!\Hom_R}(I_C , M)=0$. Then $C$ is
  unobstructed and the lower bound of Theorem~\ref{thmeulerIC} is equal to $
  \dim_{(C)}\HH(d,g)$ by Thm.\! 2.6 of \cite{krao}.
\end{remark}

\begin{remark} Let $C$ be any curve in $\proj{3}$. \\
  {\rm (i)} If $M=0$, then $\delta^1(0)=0$ and we can use Remark~\ref{thm26}
  to see that $C$ is unobstructed and that the lower bound of
  Theorem~\ref{thmeulerIC} is equal to $
  \dim_{(C)}\HH(d,g)$. 
  This coincides with \cite{El}.\\
  {\rm (ii)} If $\diam M=1$, $\dim M = r$ and $C$ is a generic curve, then $C$
  is unobstructed by \cite{krao} Cor.\! 1.6 and the lower bound is equal to $4
  d + \delta^2(0) + r \beta_{2,c}$. Indeed $ r \beta_{1,c}=0$ for a generic
  curve by \cite{krao}, Cor.\! 4.4. Moreover in this case the ``correction''
  number $ {_{-4}\!\hom_R}(I_C , M)$ is equal to $r \beta_{1,c+4}$. Hence we
  get
  $$\dim V = 4 d + \delta^2(0) + r( \beta_{2,c} + \beta_{1,c+4}). $$
  This coincides with the dimension formula of \cite{krao}, Thm.\! 3.4. 
\end{remark}

Theorem~\ref{thmeulerIC} is a consequence of the inclusion $\HH(d,g)_{\gamma}
\hookrightarrow \HH(d,g)$ of schemes. One may try the same argument for the
inclusion $\HH(d,g)_{\gamma,\rho} \hookrightarrow \HH(d,g)$ since we also for
these schemes know tangent and obstruction spaces. This leads to
\begin{theorem}\label{thmeulerICM}
  Let $C$ be a curve in $\proj{3}$ and $M= H_{*}^1(\sI_C)$. Then the dimension
  of $\HH(d,g)$ at $(C)$ satisfies
\begin{equation*}
  \dim_{(C)}\HH(d,g) \ge  1+\delta^{2}(-4)-  \sum_{i=0}^2 \eext 0Ri (M,M).
\end{equation*}
Moreover if $C$ is a generic curve of a generically smooth component $V$ of $
\HH(d,g)$, then
$$\dim V = 4 d + \delta^2(0) 
- \delta^1(0)+ \delta^1(-4) - \sum_{i=0}^1 \eext 0Ri (M,M)= 1+\delta^{2}(-4)-
\sum_{i=0}^1 \eext 0Ri (M,M) .$$
\end{theorem}

\begin{proof} We consider the stratification $ \HH(d,g)_{\gamma,\rho}$ of the
  Hilbert scheme $\HH(d,g)$ and the ``morphism'' $\phi :
  \HH(d,g)_{\gamma,\rho} \rightarrow E_{\rho}$: = isomorphism classes of
  $R$-modules $M$ given by mapping $(C)$ onto $ M(C)$. By \cite{MDP1}, Thm.\!
  1.5, $\phi$ is smooth, and $\HH(d,g)_{\gamma,M} := \phi^{-1} (M)$ is a
  scheme of dimension $1+\delta^{2}(-4)- {_0\!\hom}(M,M)$ (see
  Remark~\ref{remeulerIC}). If we ignore the scheme structures, we may still,
  for each curve $C$, consider the corresponding local deformation functor,
  $\phi_C$, of $\phi$ at $(C)$, defined on the category of local artinian
  k-algebras with residue field $k$. $\phi_C$ is smooth of fiber dimension as
  above by the results of \cite{MDP1}, see also \cite{krao}, Rem.\! 2.12 for
  the curve case and Theorem~\ref{t1.1} of this paper for the corresponding
  result for surfaces.

  It is well known that $  \EExt 0Ri (M,M)$ for $i=1,2$, determine the local
  graded deformation functor, $Def_M$, of the $R$-module $M:=M(C)$, e.g.
$$ {_0\!\ext}^1(M,M) - {_0\!\ext}^2(M,M) \le \dim  E_{\rho,M} \le \ 
{_0\!\ext}^1(M,M),$$ where $ E_{\rho,M}$ is the hull of $Def_M$ (\cite{L1},
Thm.\! 4.2.4). Moreover we have equality to the right if and only if $Def_M$
is formally smooth. Combining with the smoothness of $\phi_C$ and its fiber
dimension we get
\begin{equation} \label{equ}
  1+\delta^{2}(-4)-  \sum_{i=0}^2  {_0\!\ext}^i(M,M) \le
  \dim_{(C)} \HH(d,g)_{\gamma,\rho} \le 1+\delta^{2}(-4)-
  {_0\!\hom}(M,M)+{_0\!\ext}^1(M,M) 
\end{equation}
with equality to the right if and only if $ \HH(d,g)_{\gamma,\rho}$ is smooth
at $(C)$. This proves the inequality of the theorem since $ \dim_{(C)}\HH(d,g)
\ge \dim_{(C)} \HH(d,g)_{\gamma,\rho}$. We also get the final statement
because, at a generic curve $C$ with postulation $\gamma$ and deficiency
$\rho$, $ \HH(d,g)_{\gamma,\rho} \cong \HH(d,g)$ around $(C)$! Indeed if we
have $ \dim_{(C)} \HH(d,g)_{\gamma,\rho} < \dim_{(C)} \HH(d,g)$, then a small
neighborhood of $(C)$ in $ \HH(d,g)_{\gamma,\rho}$ is not {\it open} in $
\HH(d,g)$, contradicting the assumption that $C$ is generic in $ \HH(d,g)$.
Hence we have equality in dimensions and in fact a local isomorphism (e.g. by
generic flatness) since $ \HH(d,g)$ is smooth at $(C)$. It follows that $
\HH(d,g)_{\gamma,\rho}$ is smooth at $(C)$ and the inequality of \eqref{equ}
to the right is an equality.
\end{proof}

\begin{remark} Let $ T_{\gamma,\rho} $ be the tangent space of $
  \HH(d,g)_{\gamma,\rho}$ at $(C)$. Then we easily see from the proof that
  the upper bound in \eqref{equ} is equal to $ \dim T_{\gamma,\rho}$.
\end{remark}

If we want to generalize Theorem~\ref{thmeulerICM} to codimension 2 subschemes
in $\proj{n+2}$, the explicit replacements of $ \sum_{i=0}^1
{_0\!\ext}^i(M,M)$ in the generalized statements seem to be very complicated.
However observing that $ \sum_{i=0}^1 {_0\!\ext}^i(M,M)$ is a biliaison
invariant (since $M$ is, up to a twist), it seems to be the following weaker
form of Theorem~\ref{thmeulerICM} and \eqref{equ} which is natural to
generalize:

\begin{remark} \label{sumex} If we define $ {\rm sumext}(C)$ and ${\rm
    obsumext}(C)$ by \ ${\rm sumext}(C) = 1+\delta^{2}(-4) - \dim
  T_{\gamma,\rho}$ \ and \ ${\rm obsumext}(C) = 1+\delta^{2}(-4) - \dim_{(C)}
  \HH(d,g)_{\gamma,\rho}$\! , \! then $ {\rm sumext}(C)$ and ${\rm
    obsumext}(C)$ are biliaison invariants. We have $ {\rm sumext}(C) \le {\rm
    obsumext}(C)$
  and the equality holds if and only if $ \HH(d,g)_{\gamma,\rho}$ is smooth at
  $(C)$. 
  Furthermore if $C$ is unobstructed and generic in $ \HH(d,g)$, then $$
  \dim_{(C)}\HH(d,g) \ = 1+\delta^{2}(-4) - {\rm sumext}(C) \ . $$
\end{remark}

We have not yet proved that ${\rm obsumext}(C)$ is a biliaison invariant, but
it will follow from later results, or from \cite{MDP1}, Thm.\! 1.5 and
Remark~\ref{remeulerIC}.

For curves we have \ \ 
\begin{equation} \label{summa} {\rm sumext}(C) = \sum_{i=0}^1 \eext 0Ri (M,M)
  \ , \ \ \ {\rm and}
\end{equation}
\begin{equation} \label{obsum}
\sum_{i=0}^1  \eext 0Ri (M,M) \le {\rm obsumext}(C) \le \sum_{i=0}^2
 \eext 0Ri (M,M)
\end{equation} which we may use to compute $ {\rm sumext}(C)$ and estimate
$ {\rm obsumext}(C) $. We may also compute these invariants somewhere in the
even liaison class, e.g. by letting $C$ be the minimal curve and computing $
\dim_{(C)} \HH(d,g)_{\gamma,\rho}\! \ , \dim T_{\gamma,\rho}$ and $
\delta^{2}(-4)$ 
in this case. If $D$ is in the even liaison class 
of $C$, $D \in  \HH_{\gamma',\rho'}$,  and if we can compute  $  
\delta_D^{2}(-4)$, then we get the dimensions of $ \HH_{\gamma',\rho'}$ and
$T_{\gamma',\rho'}$, from the biliaison invariants.

\section{The dimension and smoothness of $\HH(d,p,\pi)$.}
In this section we consider the Hilbert scheme, $\HH(d,p,\pi)$, of surfaces in
$\proj{4}$. Our goal is to see how far we can generalize the results of the
preceding section to surfaces. We will focus on the dimension and the
smoothness of the Hilbert scheme. 

To compute the dimension of the components of $\HH(d,p,\pi)$, we consider the
minimal resolution of $I=I_X$:
\begin{equation} \label{res3} 0 \to \bigoplus_{i=1}^{r_4} R(-n_{4,i}) \to
  \bigoplus_{i=1}^{r_3} R(-n_{3,i}) \to \bigoplus_{i=1}^{r_2} R(-n_{2,i}) \to
  \bigoplus_{i=1}^{r_1} R(-n_{1,i}) \to I \to 0,
\end{equation}
and the invariant $ \delta^j (v)= \delta_X^j (v)$ defined by
\begin{equation}
  \delta_X^j (v) = \sum_i h^j (\sI_X (n_{1,i}+v)) - \sum_i h^j (\sI_X
  (n_{2,i}+v))+ \sum_i h^j (\sI_X (n_{3,i}+v)) - \sum_i h^j (\sI_X
  (n_{4,i}+v)). 
\end{equation}

\begin{proposition} \label{dp0.3} Let $X$ be any surface in $\proj{4}$ of
  degree $d$ and sectional genus $\pi$. Then the following expressions are
  equal
\begin{equation} \label{11}
  \begin{gathered}
    \eext 0R1 (I,I) - \eext 0R2 (I,I) + \eext 0R3 (I,I) = 1 - \delta^0(0) =
    \chi (\sN_X) - \delta^0(-5) = \\
    \chi (\sN_X) - \delta^3(0) + \delta^2(0) - \delta^1(0) = 1 + \delta^3(-5)
    - \delta^2(-5) + \delta^1(-5).
  \end{gathered}
\end{equation}
Moreover
\begin{equation} \label{hilbpolyN}
  \chi (\sN_X(v)) = dv^2 + 5dv + 5(2d + \pi  - 1) - d^2 + 2\chi (\sO_X).
\end{equation}
\end{proposition}

\begin{proof} The first upper equality follows easily by applying $\HHom vR
  (-,I)$ (for $v=0$) to the resolution \eqref{res3} because $\Hom_R (I,I)
  \simeq R$ and because the alternating sum of the dimension of the terms in a
  complex equals the alternating sum of the dimension of its homology groups.
  Similarly we compute $\delta^0 (-5)$ which through the duality
  \eqref{mduality} leads to the alternating sum of $\eext 0{\mfm}i (I,I)$.
  Combining with \eqref{mseq}, recalling $ {\sH}om_{\sO_\proj{}} (\sI_X,\sI_X)
  \cong \sO_\proj{} $ and $ {\sE}xt_{\sO_\proj{}}^1 (\sI_X,\sI_X) \cong
  \sN_X$, we get the next equality in the first line. The other equalities
  involving $\delta^j(v)$ follow from \eqref{mspect}, \eqref{mduality} and
  \eqref{seq} as outlined in \cite{krao}, Lem\! 2.2 in the curve case. The
  surface case is technically more complicated because the spectral sequence
  of the proof, $E_2^{p,q} = \EExt vRp (I, H_\mfm^q (I))$, contains one more
  non-vanishing term. The principal parts of the proof are, however, the same,
  and we leave this part to the reader. Similarly the arguments of
  \cite{krao}, Rem\! 2.4, lead to the formula
  \begin{equation} \label{hilbpolyNX} \chi (\sN_X(v)) = \chi (\sO_X(v)) + \chi
    (\sO_X(-v-5)) - d^2
\end{equation}
for any surface $X$, from which \eqref{hilbpolyN} of
Proposition~\ref{dp0.3} easily follows provided we combine with
\eqref{hilbpolyX}. Since we do not have a reference of \eqref{hilbpolyN} in
the generality of an arbitrary surface (i.e. locally Cohen-Macaulay and
equidimensional, see Remark below) and since the arguments of \cite{krao},
Rem\! 2.4 was only sketched, we will include a proof of \eqref{hilbpolyNX}.

Firstly,  we compute $\chi(\sO_X(v))=\chi(\sO_{\proj{}}(v))-\chi(\sI_X(v))$,
$ \chi(\sO_{\proj{}}(v))={ v+4 \choose 4}$, directly from \eqref{res3} as a
large sum of binomials. Recalling that $\chi(\sO_X(v))$ is the polynomial
\eqref{hilbpolyX} of degree 2, we get
\begin{equation} \label{nijX} \sum_{j=1}^4(-1)^{j-1}r_j=1 \ , \ \
  \sum_{j=1}^4(-1)^{j-1} \sum_i n_{j,i}=0 \ \ \ {\rm and} \ \ \
  \sum_{j=1}^4(-1)^{j-1} \sum_i n_{j,i}^2=-2d \ .
 \end{equation} 
 Now as in the very first part of the proof, we apply $ \HHom vR (-,I)$ to
 \eqref{res3}. Since we get $_v\!\Ext_R^i(I , I) \cong \ H^{i-1}(\sN_X(v))$
 for $v >>0$ directly from \eqref{mspect}, \eqref{mduality} and \eqref{seq}
 and we have $\Hom_R (I,I) \simeq R$, we find
 \begin{equation} \label{lem23} {\dim R_v} ~ - ~ \chi (\sN _{X} (v)) ~ = ~
   \delta ^{0} (v) ~ = ~ \sum_{j=1}^4(-1)^{j-1} \sum_i \chi (\sI_X (n_{j,i}+v))
   ~ , ~ ~ ~ v ~ >> ~ 0 \ .
\end{equation}
By \eqref{res3}, $ \chi(\sI_{X}(-v-5))= \sum_{j=1}^4(-1)^{j-1} \sum_i \chi
(\sO_{\proj{}} (-n_{j,i}-v-5))= \sum_{j=1}^4(-1)^{j-1} \sum_i \chi
(\sO_{\proj{}} (n_{j,i}+v))$. The right hand side of \eqref{lem23} is therefore
equal to $ \chi(\sI_{X}(-v-5))- \sum_{j=1}^4(-1)^{j-1} \sum_i \chi (\sO_X
(n_{j,i}+v))$. Then we compute $\sum_{j=1}^4(-1)^{j-1} \sum_i \chi (\sO_X
(n_{j,i}+v))$ by just using \eqref{hilbpolyX} and \eqref{nijX}. We get exactly
$$\sum_{j=1}^4(-1)^{j-1} \sum_i \chi (\sO_X (n_{j,i}+v)) = \chi(\sO_X(v))-d^2 ,
$$ and \eqref{lem23} translates to $ {\dim R_v} - \chi (\sN _{X} (v))  =
\chi(\sI_{X}(-v-5))- \chi(\sO_X(v))+d^2$ and we get \eqref{hilbpolyNX}.
\end{proof}
\begin{remark} Note that the formula \eqref{hilbpolyN} of
  Proposition~\ref{dp0.3} is certainly straightforward to prove for smooth
  surfaces by combining the well known formula
\begin{equation*}
  \chi (\sN_X(v)) = dv^2 + 5dv + 5(d - \pi  + 1) - 2K^2 + 14\chi (\sO_X)
\end{equation*}
with the double point formula $d^2 - 10d - 5H.K - 2K^2 + 12\chi (\sO_X) = 0$.
\end{remark}

Now we come to the analogue of Theorem~\ref{thmeulerIC}. Also in this case $
\eext 0R1 (I,I) - \eext 0R2 (I,I)$ is a lower bound of
$\HH(d,p,\pi)_{\gamma}$. Since the basic part of the proof of the Theorem
below is similar to the proof of Theorem~\ref{thmeulerIC}, we will only sketch
the proof. Note that in the surface case, we do not succeed so nicely as in
the curve case because the lower bound above is not directly given by the
first equality of Proposition~\ref{dp0.3}, due to the term $ \eext 0R3 (I,I)$.
Since we have $\EExt 0R3 (I,I) \cong {_{-5}\!\Ext_{\mathfrak m}^2}(I,I)^{\vee}
\cong {_{-5}\!\Hom_R}(I, M_1)^{\vee}$ by \eqref{mspect} and \eqref{mduality}
and $M_1 \cong H_{\mathfrak m}^2(I)$ we get at least
\begin{proposition}\label{thmeulerIX}
  Let $X$ be a surface in $\proj{4}$, let $M_i = H_*^i(\sI_X)$ for i = 1,2 and
  put $I=I_X$ and $ \delta^{j}(v)= \delta_X^{j}(v)$ for any $j$ and $v$. Then
  the dimension of $\HH(d,p,\pi)$ at $(X)$ satisfies
\begin{equation*}
  \dim_{(X)}\HH(d,p,\pi) \ge  1 + \delta^3(-5)
  - \delta^2(-5) + \delta^1(-5)- \sum_i h^1(\sI_X (n_{1,i}-5)).
\end{equation*}
Moreover let $X$ be a generic surface of a generically smooth component $V$ of
$ \HH(d,p,\pi)$ and suppose $ {_{-5}\!\Hom_R}(I , M_2)=0$. Then
$$\dim V =  1 + \delta^3(-5)
- \delta^2(-5) + \delta^1(-5)- \sum_{i=0}^1 \eext {-5}Ri(I, M_1).
$$ 
\end{proposition}
\begin{proof} For the inequality, we remark that $\eext 0R3 (I,I) =
  {_{-5}\!\hom_R}(I, M_1) \le \sum_i h^1(\sI_X (n_{1,i}-5))$ because $
  {_{-5}\!\Hom_R}(I, M_1)$ is the kernel of the map $ \oplus_i
  H^1(\sI_X(n_{1,i}-5)) \rightarrow \oplus_i H^1(\sI_X(n_{2,i}-5))$ induced by
  the corresponding map in \eqref{res3}. We conclude by
  Proposition~\ref{dp0.3}.

  To find $\dim V $ we proceed as in the proof of Theorem~\ref{thmeulerIC}
  (see the last part of the proof of Theorem~\ref{thmeulerICM} for a close
  idea), and we get $\dim V =\eext 0R1 (I,I)$, i.e. $$\dim V = 1 +
  \delta^3(-5) - \delta^2(-5) + \delta^1(-5) + \eext 0R2 (I,I) - \eext 0R3
  (I,I).$$ By \eqref{mduality} we have $ \eext 0R2 (I,I)= \eext {-5}{\mathfrak
    m}3 (I,I)$ and we conclude by the exact sequence associated to
  \eqref{mspect},
  \begin{equation} \label{lemm25} 0 \rightarrow \EExt {-5}R1 (I, H_{\mathfrak
      m}^2(I)) \rightarrow {_{-5}\!\Ext_{\mathfrak m}^3}(I,I) \rightarrow
    {_{-5}\!\Hom_R}(I, H_{\mathfrak m}^3(I)) \rightarrow {_{-5}\!\Ext _R^2}(I,
    H_{\mathfrak m}^2(I)) \rightarrow \ \ . \end{equation}\\[-8mm]
\end{proof} 
Under more specific assumptions we are able to prove,
\begin{proposition} \label{them} Let $X$ be any surface in $\proj{4}$ and
  suppose $$ {_{0}\!\Hom_R}(I, M_1)= \EExt {-5}R1(I, M_1) ={_{-5}\!\Hom_R}(I ,
  M_2)=0.$$ Then $X$ is unobstructed and $$ \dim_{(X)}\HH(d,p,\pi) = 1 +
  \delta^3(-5) - \delta^2(-5) + \delta^1(-5)- {_{-5}\!\hom_R}(I, M_1).$$ 
\end{proposition}

\begin{proof} Due to \cite{K79}, Rem.\! 3.7 (cf. \cite{W1}, Thm.\! 2.1),
  $\HH(d,p,\pi)_{\gamma} \cong \HH(d,p,\pi)$ at $(X)$ provided $
  {_{0}\!\Hom_R}(I,M_1) = 0$. Then we see by the arguments of \eqref{lemm25}
  that $ \EExt 0R2 (I,I)=0$. It follows that $\HH(d,p,\pi)_{\gamma}$ is smooth
  at $(X)$ of dimension $ \eext 0R1 (I,I)$. Then we conclude by
  Proposition~\ref{dp0.3}.
\end{proof}
\begin{remark} \label{H1NX} {\rm (i)} Proposition~\ref{them} is mainly proved
  in \cite{K5}, sect. 1. In \cite{K5} we moreover use \eqref{mspect} and
  \eqref{mduality} to prove a vanishing result for $H^1(\sN_X)$.  Indeed we
  show that $H^1(\sN_X)=0$ provided
 $$ H^1(\sI_X (n_{2,i}))= H^1(\sI_X (n_{2,i}-5))=0 \ {\rm and} \
 H^2(\sI_X(n_{1,i}))= H^2(\sI_X (n_{1,i}-5)) =0 \ {\rm for \ every} \ i. $$

 {\rm (ii)} Let $X$ be an arithmetically Cohen-Macaulay surface in $\proj{4}$.
 Then $M_1= M_2=0$ and $\delta^1(v)=\delta^2(v)=0$ for every $v$ and we can
 use Proposition~\ref{them} to see that $X$ is unobstructed and $
 \dim_{(X)}\HH(d,p,\pi) = 1 + \delta^3(-5)=1 -
 \delta^0(0)$. 
 This coincides with \cite{El}.
\end{remark}
We will illustrate the results of this section by an example. If the
assumptions of Proposition~\ref{them} or Remark~\ref{H1NX} are not satisfied,
then the surface may be obstructed, and we refer to section 8 for such
examples.
\begin{example} \label{e3.10} Let $X$ be the smooth rational surface with
  invariants $d = 11$, $\pi = 11$ (no 6-secant) and $K^2 = -11$ (cf. \cite{P}
  or \cite{DES}, B1.17, see also \cite{BoRa}). In this case the graded modules
  $M_i \simeq \oplus H^i(\sI_X(v))$ are supported at two consecutive degrees
  and satisfy
\begin{align*}
  \dim H^1 (\sI_X(3)) & = 2, & \dim H^2 (\sI_X(1)) & = 3, \\
   \dim H^1 (\sI_X(4)) & = 1, & \dim H^2 (\sI_X(2)) & = 1.
\end{align*}
Moreover $I = I_X$ admits a minimal resolution (cf. \cite{DES})
\begin{equation*}
  0 \to R(-9) \to R(-8)^{\oplus 3} \oplus R(-7)^{\oplus 3} \to R(-7)^{\oplus
    2} \oplus
  R(-6)^{\oplus 12} \to R(-5)^{\oplus 10} \to I \to 0 .
\end{equation*}
It follows that $\HHom {-5}R (I, M_2) = 0$ and $\EExt {-5}{R}i (I, M_1) = 0$
for $i = 0,1$. By Proposition~\ref{them}, $\HH(d,p,\pi)$ is smooth
at $(X)$ and $\dim_{(X)} \HH(d,p,\pi) =$
\begin{gather*} \label{} 1+ \delta ^{3} (-5)- \delta ^{3} (-5)+ \delta ^{1}
  (-5)=1+12 h^2 (\sI_X(1))-h^2 (\sI_X(2))+3 h^1 (\sI_X(3))-h^1 (\sI_X(4))=41.
  \end{gather*}
In this example it is, however, easier to use Proposition~\ref{dp0.3} to get
\begin{equation*}
  1+\delta^3(-5)-\delta^2(-5)+\delta^1(-5) = \chi
  (\sN_X)-\delta^3(0)+\delta^2(0)-\delta^1(0) = 5(2d+\pi -1)-d^2+2\chi (\sO_X)
  = 41 
\end{equation*}
because $\delta^i(0)$ for $i > 0$ is easily seen to be zero. We may also use
Remark~\ref{H1NX} to see $H^1(\sN_X)=0$. Since any smooth surface
satisfies $$H^2(\sN_X)=0 \ \ \ \ {\rm provided} \ \ \ \ H^2(\sO_X(1))=0 $$
(due to the existence of the natural surjection $\sO_X(1)^5 \to \sN_X$), we
may conclude as above directly from $ \dim H^0(\sN_X) = \chi (\sN_X) = 41$.
\end{example}
One may hope that a generalization of Theorem~\ref{thmeulerICM} to surfaces
will contain a more complete result. To do it we need to generalize some of
the theorems in \cite{MDP1} to surfaces. This will be done in the next two
sections. The biliaison statements of Remark~\ref{sumex} will be generalized
to any codimension 2 lCM equidimensional subscheme of $\proj{n+2}$ and carried
out in later sections.

\section{The smoothness of the ``morphism'' $\varphi : \HH_{\gamma,\rho}
  \to \VV_\rho$.}

In this section we prove the local smoothness of the ``morphism'' $\varphi :
\HH_{\gamma,\rho} \to \VV_\rho$ = isomorphism classes of graded $R$-modules
$M_1$ and $M_2$ satisfying $\dim (M_i)_v = \rho_i(v)$ and commuting with $b$,
given by sending the surface $X$ onto the class of the triple $(M_1,M_2,b)$
where $M_i = H_*^i(\sI_X)$ and $b \in \EExt 0R2 (M_{2},M_{1})$ is the
extension determined by $X$ (cf. Remark~\ref{r0.2} (ii)). To prove our theorem
we first take in Proposition~\ref{p0.1} a close look to Bolondi's short exact
``resolution'' of the homogeneous ideal of a surface $X$ (\cite{B2}) and how
we can define the extension $b$ given in Horrock's paper \cite{Ho}. As in
\cite{DES} the ideal is the cokernel of some syzygy modules of $M_1$ and
$M_2$, up to direct free factors. The proposition somehow uses and extends a
result of Rao for a curve $C$, namely that the minimal resolution of $I_C$ can
be put in the following form
\begin{equation} \label{resoluMI2} 0 \rightarrow L_4 \stackrel{\sigma \oplus
    0}{\longrightarrow} L_3 \oplus F_2 \rightarrow F_1 \rightarrow I_C
  \rightarrow 0 \  \
\end{equation}
where $ 0 \rightarrow L_4 \stackrel{\sigma}{\rightarrow} L_3 \rightarrow ...
\rightarrow M \rightarrow 0$ is a minimal resolution of $M$ and $F_i$ are free
modules (\cite{R}). Moreover we use local flatness criteria to generalize
Bolondi's construction in \cite{B2} so that it works for flat resolutions over
a local ring, rather than over a field. This is also the approach of
\cite{HMDP} in the curve case.

Let $X$ be a surface in $\proj{4}$ and let
\begin{equation} \label{1}
  \begin{gathered}
    0 \to P_5 \xrightarrow{\,\sigma_5\,} P_4 \xrightarrow{\,\sigma_4\,} P_3
    \xrightarrow{\,\sigma_3\,} \dots \longrightarrow P_0
    \xrightarrow{\,\sigma_0\,} M_1 \to 0, \\
    0 \to Q_5 \xrightarrow{\,\tau_5\,} Q_4 \xrightarrow{\,\tau_4\,} Q_3
    \xrightarrow{\,\tau_3\,} \dots \longrightarrow Q_0
    \xrightarrow{\,\tau_0\,} M_2 \to 0
  \end{gathered}
\end{equation}
(for short $\sigma_\bullet : P_\bullet \to M_1 \to 0$ and $\tau_\bullet :
Q_\bullet \to M_2)$ be minimal free resolutions over $R$. Let $K_\bullet$ and
$L_\bullet$ be the {\it i}th syzygies of $M_1$ and $M_2$ respectively, i.e.
$K_i = \ker \sigma_i$ and $L_i = \ker \tau_i$. Recall that syzygies have nice
cohomological properties (\cite{DES}, \cite{B2}), for instance
\begin{equation} \label{2}
  \begin{gathered}
    M_1 = H_*^1(\tilde K_1) \quad\text{and}\quad H_*^2(\tilde K_1)
    = H_*^3(\tilde K_1)  = 0, \\
    M_2 = H_*^3(\tilde L_3) \quad\text{and}\quad H_*^1(\tilde L_3)
    = H_*^2(\tilde L_3) = 0.
  \end{gathered}
\end{equation}  
There is a strong connection between the resolutions \eqref{1}, the minimal
resolution \eqref{res3} of $I = I_X$
%
%
and the following minimal resolutions of $A = H_*^0(\sO_X)$;
\begin{equation} \label{4}
  0 \to P_3' \xrightarrow{\,\sigma_3'\,} P_2' \xrightarrow{\,\sigma_2'\,}  P_1'
  \xrightarrow{\,\sigma_1'\,} P_0 \oplus  R  \to A \to 0
\end{equation}
where the morphism $P_0 \oplus R \to A$ of \eqref{4} is naturally deduced from
$P_0 \to M_1$ of \eqref{1} and the exact sequence $R \to A \to M_1 \to 0$ and
where $\sigma_\bullet' : P_\bullet' \to \ker(P_0 \oplus R \to A) \to 0$ is a
minimal $R$-free resolution (cf. \cite{MDP1}, p.\! 46). The connection we have
in mind can be formulated and proved for a {\it family} of surfaces with
constant cohomology, at least locally, e.g. we can replace the field $k$ by a
local $k$-algebra $S$. 
Now, in \cite{B2}, Bolondi uses some ideas of Horrocks \cite{Ho} to define an
element $b \in \EExt 0R2 (M_2,M_1)$ and the ``Horrocks triple''
$D=:(M_1,M_2,b)$ associated to $X$ such that, conversely given $D =
(M_1,M_2,b)$ where $M_i$ are $R$-modules of finite length, there is a surface
$X$ whose homogeneous ideal $I$ is defined in the following way. For some
integer $h \in \mathbb Z$ there is an exact sequence $0 \to L_3' \to K_1' \to
I(h) \to 0$ where $L_3'$ (resp. $K_1'$) is isomorphic to the syzygy $L_3$
(resp. $K_1$) up to some $R$-free module $F_L$ (resp. $F_K$). Up to biliaison
this construction is the inverse to the first approach which defines
$(M_1,M_2,b)$ from a given $X$. To prove the main smoothness theorem of this
section in an easy way, we need to adapt the approach above by determining
$F_L$ and $F_K$ more explicitly and such that it works over $S$. Using also
ideas of Rao's paper \cite{R}, we can prove

\begin{proposition} \label{p0.1} Let $X$ be a surface in $\proj{4}_S$, flat
  over a local noetherian $k$-algebra $S$, and suppose $M_1 = H_*^1(\sI_X)$,
  $M_2 = H_*^2(\sI_X)$ and $I=I_X$ are flat $S$-modules. Then there exist
  minimal $R$-free resolutions of $M_i$, $I$ and $A = H_*^0(\sO_X)$ as in
  \eqref{1}, \eqref{res3} and \eqref{4}, with $R = S[X_0,X_1,..,X_4]$.
  Moreover let $L_3'= \ker \sigma_1'$ and let $K_1'$ be the kernel of the
  composition of $\sigma_1'$ and the natural projection $P_0 \oplus R \to
  P_0$, cf. \eqref{4}. Then there is an exact sequence
  \begin{equation} \label{5} 
   0 \to L_3'\xrightarrow{\,b'\,} K_1' \to I \to 0
\end{equation}
of flat graded $S$-modules and a surjective morphism $d : \HHom {0}{R} (L_3',
K_1') \to \EExt 0R2 (M_2,M_1)$, defining a triple $(M_1,M_2,b)$ where $b =
d(b')$, coinciding with the uniquely defined ``Horrocks triple'' of \cite{Ho}
or \cite{B2}. Moreover $L_3'$ (resp. $K_1'$) is the direct sum of a 3rd
syzygy of $M_2$ (resp. 1st syzygy of $M_1$) up to a direct free factor, i.e.
there exist $R$-free modules $F_L$ and $F_K$ such that the horizontal exact
sequences in the diagram
\begin{gather*}
  0 \longrightarrow  K_1' \hspace{10pt} \longrightarrow
  \hspace{10pt}
  P_1'   \hspace{18pt}   \longrightarrow   \hspace{10pt}   P_0 \\
  \hspace{25pt} \downarrow \hspace{22pt} \circ \hspace{22pt} \downarrow
  \hspace{22pt}   \circ  \hspace{28pt}    \| \\
  0 \to K_1 \oplus F_K \to P_1 \oplus F_K \xrightarrow{\sigma_1 \oplus 0} P_0
\end{gather*}
are isomorphic (i.e., the downarrows are isomorphisms). Similarly, the exact
sequences $0 \to Q_5 
\stackrel{(\tau_5 ,0)}{\longrightarrow} Q_4 \oplus F_L \to L_3 \oplus F_L \to
0$ and $0 \to P_3' \to P_2' \to L_3' \to 0$ are isomorphic as well.
\end{proposition}

\begin{remark} \label{r0.2} {\rm (i)} By a surface $X \subseteq \proj{4}_S$ in
  Proposition~\ref{p0.1} we actually mean that $X \times_{\Spec(S)} \Spec(k)$
  is a surface (i.e. locally Cohen-Macaulay and equidimensional of dimension
  2).

  {\rm (ii)} The proposition above, defining the ``Horrocks triple''
  $(M_1,M_2,b)$ from a given $X \subseteq \proj{4}_S$, can be regarded as our
  definition of the ``morphism'' $\varphi : \HH_{\gamma,\rho} \to \VV_\rho$ =
  isomorphism classes of graded $R$-modules $M_1$ and $M_2$ satisfying $\dim
  (M_i)_v = \rho_i(v)$ and commuting with $b$.
\end{remark}

\begin{proof} We obviously have minimal resolutions of $M_i\otimes_S k$,
  $I_X\otimes_S k$ and $A\otimes_S k$ as described above with $R =
  k[X_0,X_1,..,X_4]$, cf. \eqref{1}, \eqref{res3} and \eqref{4}. These
  resolutions can easily be lifted to the minimal resolution of the
  proposition by cutting into short exact sequences and using the flatness of
  the modules involved.

  By the definition of $L_3'$ and $K_1'$ there is a commutative diagram
\begin{gather*}
  \hspace{105pt} 0 \hspace{5pt} \longrightarrow \hspace{5pt} R \hspace{8pt}
  \longrightarrow \hspace{5pt}R
  \longrightarrow 0 \\
  \hspace{117pt}
  \downarrow  \hspace{15pt}    \circ  \hspace{15pt}    \downarrow \\
  \hspace{32pt} 0 \longrightarrow L_3' \hspace{3pt}  \longrightarrow P_1'
  \longrightarrow P_0  \oplus  R  \longrightarrow A  \longrightarrow 0 \\
  \hspace{35pt} \downarrow \hspace{12pt} \circ \hspace{12pt} \| \hspace{15pt}
  \circ \hspace{15pt}
  \downarrow  \hspace{15pt}    \circ  \hspace{15pt}    \downarrow \\
  \hspace{38pt} 0 \longrightarrow K_1' \longrightarrow P_1' \hspace{4pt}
  \longrightarrow \hspace{4pt} P_0 \hspace{4pt} \longrightarrow \hspace{8pt}
  M_1 \longrightarrow 0
\end{gather*}
and we get easily the exact sequence \eqref{5} by the snake lemma. Comparing
the lower exact sequence in the last diagram with the following part of the
minimal resolution of $M_1$; $\to P_1 \to P_0 \to M_1 \to 0$, we get the
commutative diagram of the proposition because $K_1$ is the 1st syzygy of
$M_1$.

To prove the corresponding commutative diagram for $L_3'$ and $L_3$, we
sheafify \eqref{5}, and we get $M_2 \simeq H_{*}^3(\tilde L_3')$. Recalling
the definition of $L_3'$, we get the exact sequence
\begin{equation*}
  H_{*}^4(\tilde P_2')^{\vee} \to H_{*}^4(\tilde P_3')^{\vee} \to M_2^{\vee}
  \simeq  \Ext_R^5 (M_2,R(-5)) \to 0  
\end{equation*}
which we compare to the {\it minimal} resolution
\begin{equation*}
  Q_4^{\vee} \to Q_5^{\vee} \to \Ext_R^5 (M_2, R) \to 0
\end{equation*}
obtained by applying $\Hom_R (-,R)$ to the resolution $Q_\bullet \to M_2$.
Recalling $H_*^4(\tilde P_i')^{\vee} (5) \simeq P_i'^{\vee}$, we get the
conclusion, as in the proof of Thm.\! 2.5 of \cite{R}.

Finally to define the morphism $d$ and to see that the defined triple
$(M_1,M_2,b)$ is the one given by Horrocks' construction (seen to be unique by
\cite{Ho}), one may consult \cite{B2} for the case $S = k$ which, however,
easily generalizes to a local ring $S$. The important part is as follows. The
definition of $K_1'$ and $K_0$ implies $\Ext^2 (M_2,M_1) \simeq \Ext^3
(M_2,K_0) \simeq \Ext^4 (M_2,K_1')$. Next, by Gorenstein duality, we know
$\Ext_R^i (M_2,R) = 0$ for $i \neq 5$. Hence the definition of the syzygies
$L_i$ leads to $\Ext^4 (M_2,K_1') \simeq \Ext^3 (L_0,K_1') \simeq
\Ext^1 (L_2,K_1')$ and to a diagram 
\begin{equation} \label{6}
  \begin{gathered}
    \HHom 0R (Q_3,K_1') \to \HHom 0{} (L_3,K_1') \to \EExt 0{}1 (L_2, K_1')
    \to 0 \\
   \hspace{62pt}  \downarrow  \hspace{80pt} \downarrow \\
   \hspace{80pt}  \HHom 0{} (L_3',K_1') \hspace{18pt} \EExt 0R2 (M_2,M_1)
  \end{gathered}
\end{equation}
where the horizontal sequence is exact and the first (resp. second) vertical
map is injective and split (resp. an isomorphism). We let $d : \HHom {0}{R}
(L_3', K_1') \to \EExt 0R2 (M_2,M_1)$ be the obvious composition, first using
the ``inverse'' of the split map, and we get the conclusions of the
proposition.
\end{proof}

Now we will show the smoothness of $\varphi$. Indeed
Proposition~\ref{p0.1} allows a rather easy proof of

\begin{theorem} \label{t1.1} The ``morphism'' $\varphi : \HH_{\gamma,\rho} \to
  \VV_\rho$ = isomorphism classes of graded $R$-modules $M_1$ and $M_2$
  satisfying $\dim (M_i)_v = \rho_i(v)$ and commuting with $b$, is smooth
  (i.e. for any surface $X$ in $\proj{4}_k$, the corresponding local
  deformation functor of $\varphi$, 
  given by $(X_S \subseteq \proj{4}_S) \mapsto {\rm class \ of \ }
  (M_{1S},M_{2S},b_S)$, see right below, is formally smooth).
\end{theorem}

\begin{proof} Let $T \to S \to k$ be surjections of local artinian $k$-algebras
  with residue fields $k$ such that $\ker (T \to S)$ is a $k$-module via $T
  \to k$. Let $X_S \subseteq \proj{4}_S$ be a deformation of $X \subseteq
  \proj{4}$ to $S$ with constant postulation $\gamma$ and constant deficiency
  $\rho = (\rho_1,\rho_2)$. Let $(M_{1S},M_{2S},b_S)$ be the ``Horrocks
  triple'' defined by $X_S$ (cf. Proposition~\ref{p0.1}). Note that $M_{iS}$
  for $i=1,2$ are $S$-flat since $ \rho_i$ are constant. Let
  $(M_{1T},M_{2T},b_T)$ be a given deformation of $(M_{1S},M_{2S},b_S)$ to
  $T$. To prove the smoothness at $(X)$, we must show the existence of a
  deformation $X_T \subseteq \proj{4}_T$ of $X_S \subseteq \proj{4}_S$, whose
  corresponding ``Horrocks triple'' is precisely $(M_{1T},M_{2T},b_T)$, modulo
  graded isomorphisms of  $(M_{1T},M_{2T})$ commuting with $b_T$.

  We have by Proposition~\ref{p0.1} minimal resolutions of $M_{iS}$, $I_{X_S}$
  and $A_S$ over $R_S = S[X_0,X_1,..,X_4]$ as in \eqref{res3},
  \eqref{1}-\eqref{4} and flat $S$-modules $L_{iS}$, $K_{iS}$, $L_{3S}'$,
  $K_{1S}'$ fitting into the exact sequence \eqref{5} and a surjection $d$
  defined as the composition (cf. \eqref{6})
\begin{equation} \label{12}
  \begin{gathered}
    \HHom 0{R_S} (L_{3S}',K_{1S}') \to \HHom 0{R_S} (L_{3S}, K_{1S}') \to
    \EExt 0{R_S}1 (L_{2S}, K_{1S}') \simeq \EExt 0{R_S}2 (M_{2S}, M_{1S}) \\
    \hspace{10pt} \inup
    \hspace{90pt}  \inup  \hspace{90pt} \inup  \hspace{90pt}  \inup \\
    \hspace{10pt} b_S' \hspace{30pt} \longrightarrow \hspace{30pt} \beta_S
    \hspace{30pt} \longrightarrow \hspace{30pt} b_S \hspace{30pt}
    \longrightarrow \hspace{30pt} b_S
  \end{gathered}
\end{equation}
``on the $S$-level'' ($\beta_S$ is simply the image of $b_S'$ via the map of
\eqref{12}) which lifts the corresponding resolutions/modules/sequences on the
``$k$-level''. Since $M_{iT}$ are given deformations of $M_{iS}$, we can lift
the minimal resolutions $\sigma_{\bullet S} : P_{\bullet S} \to M_{1S}$ and
$\tau_{\bullet S} : Q_{\bullet S} \to M_{2S}$ further to $T$, thus proving the
existence of deformations $L_{iT}$, $K_{iT}$, $L_{3T}'$, $K_{1T}'$ of
$L_{iS}$, $K_{iS}$, $L_{3S}'$, $K_{1S}'$ resp. (the free submodules $F_{LS}$
and $F_{KS}$ of $L_{3S}'$ and $K_{1S}'$ are lifted trivially). So we have a
diagram \eqref{6} and hence a sequence \eqref{12} ``on the $T$-level'' where
the elements $b_{T}'$ and $\beta_T$ are not yet defined. The element $b_T \in
\EExt 0{}1 (L_{2T}, K_{1T}') \simeq \EExt 0{R_T}2 (M_{2T},M_{1T})$ is, however,
given and if we consider the diagram (cf. \eqref{6})
\begin{gather*}
  \HHom 0{R_T} (Q_{3T},K_{1T}') \to \HHom 0{R_T} (L_{3T}, K_{1T}') \to \EExt
  0{R_T}1 (L_{2T}, K_{1T}')
  \to 0 \\
  \downarrow \hspace{45pt} \circ \hspace{45pt} \downarrow \alpha
  \hspace{35pt} \circ \hspace{40pt} \downarrow \\
  \HHom 0{R_S} (Q_{3S},K_{1S}') \to \HHom 0{R_S} (L_{3S}, K_{1S}') \to \EExt
  0{R_S}1 (L_{2S}, K_{1S}') \to 0
\end{gather*}
of exact horizontal sequences and surjective vertical maps deduced from $0 \to
L_{3T} \to Q_{3T} \to L_{2T} \to 0$, we easily get a morphism $\beta_T \in
\HHom 0{} (L_{3T}, K_{1T}')$ such that $\alpha (\beta_T) = \beta_S$, i.e.,
$\beta_T \otimes_T S = \beta_S$. Since $L_{3S}' \simeq L_{3S} \oplus F_{LS}$
we can decompose the map $b_S'$ as $(\beta_S,\gamma_S) \in \HHom 0{} (L_{3S}',
K_{1S}')$, and taking any lifting $\gamma_T : F_{LT} \to K_{1T}'$ of
$\gamma_S$, we get a map $b_{T}'= (\beta_T,\gamma_T) \in \HHom 0{} (L_{3T}',
K_{1T}')$ fitting into a commutative diagram
\begin{align*}
  L_{3T} \oplus  F_{LT} \simeq L_{3T}' & \xrightarrow{b_{T}'} K_{1T}' \\
  \downarrow \hspace{7pt} & \hspace{7pt} \circ \hspace{12pt} \downarrow \\
  L_{3S} \oplus  F_{LS} \simeq L_{3S}' & \xrightarrow{b_S'} K_{1S}' \ .
\end{align*}
%
%
Once having proved the existence of such a commutative diagram, we can define
a surface $X_T$ of $\proj{4}_T$ with the desired properties, thus proving the
claimed smoothness. Indeed it is straightforward to see that $\coker b_{T}'$
is a (flat) deformation of $\coker b_S'= I_{X_S}$ to $T$. Moreover one knows
that an $R_T = T[X_0,X_1,..,X_4]$-{\it module} $\coker b_{T}'$ which lifts a
graded ideal $I{X_S}$ is again a {\it graded ideal} $I_T$ (we can deduce this
information by interpreting the isomorphisms $H^{i-1}(\sN_X) \simeq \EExt
{}{\sO_\proj{}}i (\sI_X, \sI_X)$ for $i = 1,2$ in terms of their deformation
theory from which we see that $ {\widetilde \coker b_{T}'}$ is a sheaf ideal,
and we conclude by taking global sections, cf. \cite{W1} or \cite{krao},
Lem.\! 4.8  for further details). Hence we have proved the
existence of a surface $X_T = \Proj (R_T/I_T)$, flat over $T$ which via $T \to
S$ reduces to $X_S$. By the construction above the corresponding ``Horrocks
triple'' is precisely the given triple $(M_{1T},M_{2T},b_T)$, and we are done.
\end{proof}

\begin{corollary} \label{p1.3} Let $X$ be a surface in $\proj{4}$. If the
  local deformation functors $Def(M_i)$ of $M_i$ are formally smooth (for
  instance if $\EExt 0R2 (M_i, M_i) = 0$) for $i = 1,2$, and if
\begin{equation*}
  \EExt 0R3 (M_2, M_1) = 0 ,
\end{equation*}
then $ \HH_{\gamma,\rho}$ is smooth at $(X)$. 
\end{corollary}

\begin{proof}
  With notations as in the very first part of the proof of Theorem~\ref{t1.1},
  it suffices to prove that there always exists a deformation
  $(M_{1T},M_{2T},b_T)$ of $(M_{1S},M_{2S},b_S)$ since then the proof above
  shows the existence of a deformation $X_T = \Proj (R_T/I_T)$ which reduces
  to $X_S$ via $T \to S$. Since $Def(M_i)$ are formally smooth, it suffices to
  show the existence of $b_T$ which maps to $b_S \in \EExt 0{R_S}2 (M_{2S},
  M_{1S})$. Let $\mathfrak{a}= \ker (T \to S)$. If we apply $ \HHom 0{R_T}
  (M_{2T},-)$ to the exact sequence $$0 \to \mathfrak{a} \otimes_T M_{1T}
  \cong \mathfrak{a} \otimes_k M_{1} \to M_{1T} \to M_{1S} \to 0$$ and use
  that $ \EExt 0R3 (M_2, M_1) = 0$, we see that $\EExt 0{R_T}2 (M_{2T},
  M_{1T}) \to \EExt 0{R_T}2 (M_{2T}, M_{1S})$ is surjective. Hence we get a
  surjective map $\EExt 0{R_T}1 (L_{3T}, K_{1T}') \simeq \EExt 0{R_T}2
  (M_{2T}, M_{1T}) \to \EExt 0{R_S}1 (L_{2S}, K_{1S}') \simeq \EExt 0{R_S}2
  (M_{2S}, M_{1S})$ and we are done.
\end{proof}
\begin{remark} \label{r1.2} If we, as in \cite{MDP1} for curves, had proven the
  existence of the ``fiber'' $\HH_{\gamma,D}$, $D = (M_1,M_2,b)$, of $\varphi$
  as a scheme, then Theorem~\ref{t1.1} must imply the smoothness of
  $\HH_{\gamma,D} $ while \cite{BM1} implies its irreducibility.
  Indeed \cite{BM1}, cor. 3.2 tells that the family of surfaces in $\proj{4}$
  belonging to the same shift of the same liaison class, with fixed
  postulation, form an irreducible family, from which we see that
  $\HH_{\gamma,D}$ is irreducible. Note that we can work with $\HH_{\gamma,D}$
  as a locally closed subset of $\HH_{\gamma,\rho}$ (cf. the arguments of
  \cite{BB}, cor.\! 2.2, and combine with Proposition~\ref{p0.1}), even though
  we have not proved that $ \varphi$ extends to a morphism of representable
  functors.
\end{remark}

\section{The tangent space of $\HH_{\gamma,\rho}$.}

In this section we determine the tangent space of $\HH_{\gamma,\rho}$ at $(X)$
and we give a criterion for $ \HH_{\gamma,\rho} \cong \HH(d,p,\pi)$ to be
isomorphic as schemes at $(X)$. We end this section by considering an example.

Let $X$ be a surface in $\proj{4}$ with graded ideal $I = I_X$ and let $D =
(M_1,M_2,b)$, $M_i = H_*^i(\tilde I)$, be its ``Horrocks triple''. Recall
that $\EExt 0R1 (I,I)$ is the {\it tangent space} of
$\HH_\gamma$ at $(X)$ because a deformation in $\HH_\gamma$ keeps the
postulation constant, i.e. it corresponds precisely to a graded deformation of
$I$. Moreover there exist maps
\begin{equation*}
  \varphi_i : \EExt 0R1 (I, I) \to \HHom 0R (H_*^i(\tilde I),
  H_*^{i+1}(\tilde I)). 
\end{equation*}
taking an extension $0 \to I \to E \to I \to 0$ of $\EExt 0R1 (I, I)$ onto the
connecting homomorphism $\delta^i$ in the exact sequence
\begin{equation*}
  H_*^i(\tilde E) \to H_*^i(\tilde I) \xrightarrow { \delta^i}
  H_*^{i+1}(\tilde  I)  \to H_*^{i+1}(\tilde E) 
\end{equation*}
For graded homogeneous ideals we have $I = H_*^0(\tilde I)$, and it follows
that the composition $E \to H_*^0(\tilde E) \to H_*^0(\tilde I)$ is
surjective, i.e. we get $\varphi_0 = 0$. Moreover note that if $\delta^{i-1}$
and $\delta^i$ are both zero for some $i$, then the exact sequence $0 \to I
\to E \to I \to 0$ above defines an extension
\begin{equation*}
  0 \to H_*^i(\tilde I) \to H_*^i(\tilde E) \to H_*^i(\tilde I) \to 0
\end{equation*}
Since $M_i = H_*^i(\tilde I)$ for $i=1,2$ and $E= H_*^3(\tilde I)$, there are
well-defined morphisms $$\psi_i : \ker (\varphi_1,\varphi_2) \to \EExt 0R1
(M_i,M_i) \ \ \ \ {\rm for} \ \ \ \ {\rm i = 1,2} $$ where $
(\varphi_{1},\varphi_2) : \!\EExt 0R1 (I,I) \to \HHom 0{} (M_1, M_2) \times
\HHom 0{} (M_2, E)$ and $\varphi_i$ are defined above. Recalling $\rho =
(\rho_1,\rho_2)$ we put
\begin{equation} \label{tansp} \EExt 0R1 (I, I)_\rho := \ker
  (\varphi_1,\varphi_2).
\end{equation}
Using base change theorems, as in \cite{MDP1}, we easily show that $\ker
(\varphi_1,\varphi_2)$ is the tangent space of $\HH_{\gamma,\rho}$ at $(X)$,
i.e. we get

\begin{proposition} \label{dp3.1} $ \EExt 0R1 (I, I)_\rho$ is the tangent
  space of $\HH_{\gamma,\rho}$  at $(X)$. In particular if 
\begin{equation} \label{3hom}
  \HHom 0R (I, M_1) = 0, \quad \HHom 0R (M_1, M_2) = 0 \quad\text{and}\quad
  \HHom 0{} (M_2, E) = 0,
\end{equation}
then the tangent spaces of $\HH_{\gamma,\rho}, \HH_\gamma$ and $\HH(d,p,\pi)$
are isomorphic at $(X)$. Indeed $ \HH_\gamma \cong \HH(d,p,\pi)$ as schemes at
$(X)$, and if $\HH_{\gamma,\rho}$ is smooth at $(X)$, then $\HH_{\gamma,\rho}
\cong \HH_\gamma$ are isomorphic as schemes at $(X)$ as well.
\end{proposition}

\begin{proof} As earlier remarked, $\EExt 0R1 (I,I) \cong \EExt {}{}1
  (\sI_X,\sI_X) \cong H^{0}(\sN_X)$ provided $ \HHom 0R (I, M_1) = 0$.
  Moreover $ \EExt 0R1 (I, I)_\rho \cong \EExt 0R1 (I,I)$ since $\varphi_i=0$
  for $i=1,2$.

  For the isomorphism as schemes we remark that $\HH_\gamma \simeq
  \HH(d,p,\pi)$ is proven in \cite{K79}, Thm.\! 3.6 and Rem. 3.7 (see also
  \cite{W1}). Finally if $\HH_{\gamma,\rho}$ is smooth at $(X)$, then the
  embedding $\HH_{\gamma,\rho} \hookrightarrow \HH_\gamma$ is smooth at $(X)$
  (since the tangent map is surjective), hence etale, hence an isomorphism at
  $(X)$ since the embedding is universally injective.
\end{proof}

\begin{remark} \label{t3.7} If we suppose \eqref{3hom}, then
  $\HH_{\gamma,\rho} \cong \HH_\gamma$ are isomorphic as schemes at $(X)$ by
  \cite{K6}, Thm.\! 3.7 without requiring the smoothness of
  $\HH_{\gamma,\rho}$ at $(X)$. See also Remark~\ref{seminat}.
\end{remark}

In \cite{K6} we also gave almost complete proofs of Remark~\ref{t3.7} and of
the following two non-trivial results (cf. \cite{K6}, Prop.\! 3.4 and Prop.\!
3.6). Note that Remark~\ref{p3.4} generalizes Corollary~\ref{p1.3}.
\begin{remark} \label{p3.4} Let $X$ be a surface in $\proj{4}$. Then
  there exists morphisms $e_i : \EExt 0R1 (M_i, M_i) \to \EExt 0R3 (M_2, M_1)$
  for $i=1,2$ and an induced morphism
\begin{equation*}
  \bar e_1 : \EExt 0R1 (M_1, M_1) \to \EExt 0R3 (M_2, M_1)/e_2(\EExt 0R1 (M_2,
  M_2))
\end{equation*}
such that if the local deformation functors $Def(M_i)$ of $M_i$ are formally
smooth (for instance if $\EExt 0R2 (M_i, M_i) = 0$) for $i = 1,2$, and if the
morphism $\bar e_1$ is surjective, then $\VV_\rho$ is smooth at $D =
(M_1,M_2,b)$ (i.e. the local deformation functor of $D$ is formally smooth).
\end{remark}
\begin{remark} \label{p3.6} Let $X$ be a surface in $\proj{4}$ and let
  $\epsilon = \dim \coker \bar e_1$. Then $\dim \EExt 0R1 (I, I)_\rho =$
\begin{gather*} \label{} 1+ \delta ^{3} (-5)+ \sum_{ i = 0 }^{3} (-1) ^{i} ~
  {_0\!\ext_R^i} (M_2,M_1)- \sum_{i= 0 }^{1} (-1) ^{i} ~ {_0\!\ext_R^i}
  (M_1,M_1)- \sum_{i= 0 }^{1} (-1) ^{i} ~ {_0\!\ext_R^i} (M_2,M_2) +\epsilon.
\end{gather*}
\end{remark}

To illustrate the results we have proved, we consider an example of a surface
$X$ of $\proj{4}$ where actually $\VV_\rho$ is smooth and non-trivial at the
corresponding $(M_1,M_2,b)$, cf.\! Corollary~\ref{p1.3}. Moreover all
conditions of Proposition~\ref{dp3.1} are satisfied, and it follows that
$\HH_{\gamma,\rho}$ and $\HH(d,p,\pi)$ are isomorphic and smooth at $(X)$.

\begin{example} \label{e3.10new} Let $X$ be the smooth elliptic surface with
  invariants $d = 11$, $\pi = 12$ and $K^2 = -4$ (cf. \cite{P} or \cite{DES},
  B7.6). Then the graded modules $M_i \simeq \oplus H^i(\sI_X(v))$ for $i=1,2$
  vanish for every $v$ except in  the following cases
\begin{align*}
  h^1 (\sI_X(3)) = 1, \quad h^2 (\sI_X(1)) = 2, \quad  h^2 (\sI_X(2)) = 1.
\end{align*}
Moreover $I = I_X$ admits a minimal resolution (cf. \cite{DES})
\begin{equation*}
  0 \to R(-8) \to  R(-7)^{\oplus 6} \to R(-6)^{\oplus 13} \to R(-5)^{\oplus 8}
  \oplus   R(-4)^{} \to I \to 0 .
\end{equation*}
It follows that $\EExt 0{}i (M_j, M_j) = 0$ for $i \ge 2$ and $j = 1,2$ and
that $\EExt 0{}3 (M_2, M_1) = 0$. By Corollary~\ref{p1.3} and
Proposition~\ref{dp3.1} we get that $\HH(d,p,\pi) \cong \HH_{\gamma,\rho}$ is
smooth at $(X)$. If we, however, want to compute the dimension of
$\HH(d,p,\pi)$ at $(X)$ and will avoid Remark~\ref{p3.6} which we have not
proved, we still have to use the results of section 4. Let us only use the two
``most general'' results there, Proposition~\ref{dp0.3} and
Propositions~\ref{thmeulerIX}, to illustrate the principle of semicontinuity a
little extended (to include the semicontinuity of the graded Betti numbers).
Let $V$ be the generically smooth component of $ \HH(d,p,\pi)$ to which $(X)$
belongs. Since $\HH(d,p,\pi) \cong \HH_{\gamma,\rho}$ at $(X)$, then a generic
surface $\tilde X$ of $V$ also belongs to $\HH_{\gamma,\rho}$. Inside
$\HH_{\gamma}$, hence inside $\HH_{\gamma,\rho}$, the graded Betti numbers of
the homogeneous ideal of the surfaces obey semicontinuity by Remark 7(b) of
\cite{K07}!! Since we from the minimal resolution of $I_X$ can see that, for
every $i$, $\beta_{j,i} \ne 0$ for at most one $j$ and since the Hilbert
functions of $X$ and $\tilde X$ are the same, they have exactly the same
graded Betti numbers. Moreover note that $ h^i (\sI_{\tilde X}(v))= h^i
(\sI_X(v))$ for any $i,v$ since $X$ has seminatural
cohomology. 
It follows that \ $$\dim V = 1+ \delta ^{3} (-5)- \delta ^{3} (-5)+ \delta
^{1}(-5)=$$
 $$1+ h^3
 (\sI_X(-1)) + 8h^3(\sI_X)+13 h^2 (\sI_X(1))-6 h^2 (\sI_X(2))-
 h^1(\sI_X(3))=50.$$ Since we have proved $\dim V = 1+ \delta ^{3} (-5)-
 \delta ^{3} (-5)+ \delta ^{1}(-5)$ it is easier to use
 Proposition~\ref{dp0.3} to get
\begin{equation*}
  \dim V = \chi
  (\sN_X)-\delta^3(0)+\delta^2(0)-\delta^1(0) = 5(2d+\pi -1)-d^2+2\chi (\sO_X)
  = 50
\end{equation*}
because $\delta^i(0)$ for $i > 0$ is easily seen to be zero.
\end{example}

\section{Linkage of surfaces.}

The main result of this section shows how to compute the dimension of
$\HH_{\gamma, \rho}$ and the dimension of its tangent space at $(X)$ provided
we know how to solve the corresponding problem for a linked surface $X'$
(Theorem~\ref{t4.1}). In another related result (Proposition~\ref{vanii}) we
give conditions on e.g. a
generic surface of $\HH(d,p,\pi)$ 
such that corresponding linked surface is non-generic in the sense
$\dim_{(X')} \HH_{\gamma',\rho'} < \dim_{(X')} \HH (d',p',\pi')$. In this case
a new surface, the generic one with ``smaller'' cohomology, has to exist. In
proving the results of this section we substantially need the theory of
linkage of families developed in \cite{K3}.

Since the main even liaison result of this paper, which we prove in the final
section, requires that the linkage theorem of this section is proven for
equidimensional locally Cohen-Macaulay (lCM) codimension 2 subschemes of
$\proj{n+2}$, we prove Theorem~\ref{t4.1} in this generality. The other
results and examples of this section deal, however, with surfaces.

Now, if the surfaces $X$ and $X'$ are (algebraically) linked by a complete
intersection (a CI) $Y$ of type $(f,g)$, then the dualizing sheaf $\omega_{X'}$
satisfies $\omega_{ X'} = \sI_{X/Y}(f+g-5)$ where $\sI_{X/Y} = \ker (\sO_Y \to
\sO_X)$ (\cite{PS}, \cite{MIG}). Moreover $\omega_X = \sI_{X'/Y}(f+g-5)$ and
we get
\begin{equation} \label{29}
\begin{gathered}
  \chi (\sO_X(v)) + \chi (\sO_{X'}(f+g-5-v)) = \chi (\sO_Y(v)) \\[3pt]
  \begin{aligned}
    h^i(\sI_{X'}(v)) & = h^{3-i}(\sI_X(f+g-5-v)), & \text{for $i=1$ and 2} \\
    h^i(\sI_{X'/Y}(v)) & = h^{2-i}(\sO_X(f+g-5-v)), & \text{for $i=0$ and 2} \\
    h^i(\sO_{X'}(v)) & = h^{2-i}(\sI_{X/Y}(f+g-5-v)), & \text{for $i=0$ and 2}
  \end{aligned}
\end{gathered}
\end{equation}
from which we deduce $d + d' = fg$ and $\pi'- \pi = (d'- d)(f+g-4)/2$. 

The generalization of \eqref{29} to equidimensional lCM codimension 2
subschemes of $\proj{n+2}$ is clear, e.g. we have 
\begin{equation} \label{29codim2} h^i(\sI_{X'/Y}(v)) =
  h^{n-i}(\sO_X(f+g-n-3-v)), \quad \text{for $i=0$ and} \ n.
\end{equation} 
Note that we now have $n$ deficiency modules, whose dimensions $\rho_i(v)=
h^i(\sI_{X}(v))$, $i=1,2,...,n$ determine the vector function
$\rho=(\rho_1,...,\rho_n)$. Using this vector function, we easily generalize
\eqref{tansp} in such a way that we get the tangent space $\EExt 0R1 (I_X,
I_X)_\rho$ of the Hilbert scheme $\HH_{\gamma,\rho} \subseteq {\rm
  Hilb}^{p(v)}( \proj{n+2})$ of constant cohomology in this case. 
We allow $n=0$ in which case there is no $\rho$ and $\HH_{\gamma,\rho}
\subseteq {\rm Hilb}^{p(v)}( \proj{2})$ should be taken as the Hilbert scheme
of constant postulation (``the postulation Hilbert scheme'') and $\EExt 0R1
(I_X, I_X)_\rho$ as $\EExt 0R1 (I_X, I_X)$. We have (cf.\! \cite{MDP1} for the
curve case of the theorem),
    
\begin{theorem} \label{t4.1} Let $X$ and $X'$ be two equidimensional locally
  Cohen-Macaulay codimension 2 subschemes of $\proj{n+2}$, linked by a
  complete intersection $Y \subseteq \proj{n+2}$ of type $(f,g)$, and suppose
  that $(X)$ (resp. $(X')$) belongs to the Hilbert scheme
  $\HH_{\gamma,\rho}$ (resp. $\HH_{\gamma',\rho'}$) of constant
  cohomology.  Then  \\[-2mm]
	
  {\rm i)} $\dim_{(X)} \HH_{\gamma,\rho} + h^0(\sI_X(f)) + h^0(\sI_X(g)) =
  \dim_{(X')} \HH_{\gamma',\rho'} + h^0(\sI_{X'}(f))+ h^0(\sI_{X'}(g))$ or
  equivalently, \\[-3mm]

  $\dim_{(X')} \HH_{\gamma',\rho'} = \dim_{(X)} \HH_{\gamma,\rho} +
  h^0(\sI_{X/Y}(f))
  + h^0(\sI_{X/Y}(g)) - h^n(\sO_X(f-n-3)) - h^n(\sO_X(g-n-3)).$ \\[-2mm]

  {\rm ii)} The dimension formulas of i) remain true if we replace $\dim_{(X)}
  \HH_{\gamma, \rho}$ and $\dim_{(X')} \HH_{\gamma',\rho'}$ by the dimension
  of their tangent spaces $\EExt 0R1 (I_X, I_X)_\rho$ and $\EExt 0R1 (I_{X'},
  I_{X'})_{\rho '}$ respectively.  \\[-2mm]
	
 {\rm iii)} $\HH_{\gamma,\rho}$ is smooth at $(X)$ if and only if
$\HH_{\gamma',\rho'}$ is smooth at $(X')$
\end{theorem}

\begin{proof} Let $D(p(v); f,g)$ be 
  the Hilbert flag scheme parametrizing of pairs $(X,Y)$ of equidimensional
  lCM codimension 2 subschemes of $\proj{n+2}$ such that $Y$ is a CI of type
  $(f,g)$ containing $X$. By \cite{K3}, Thm.\! 2.6, there is an isomorphism of
  schemes,
  \begin{equation} \label{30} D(p(v); f,g) \simeq D(p'(v); f,g),
\end{equation}
given by sending $(X,Y)$ onto $(X',Y)$ where $X'$ is linked to $X$ by $Y$. We
may suppose $n \ge 1$ in Theorem~\ref{t4.1} since the case $n=0$ is completely
solved by Prop.\! 1.7 of \cite{K98}. Then the projection morphism $p : D(p(v);
f,g) \to {\rm Hilb}^{p(v)}( \proj{n+2})$, given by $(X,Y) \mapsto (X)$, is
smooth at $(X,Y)$ provided $H^1(\sI_X(f)) = H^1(\sI_X(g)) = 0$ (\cite{K3},
Thm.\! 1.16 (b)). By \cite{K3}, Rem.\! 1.20, see also \cite{K3}, Rem.\! 1.21,
this smoothness holds if we replace the vanishing above with the claim that
the corresponding twisted ideal sheaves on $ {\rm Hilb}^{p(v)}( \proj{n+2})$
are locally free and commute with base change. Hence the following restriction
of $p$ to $p^{-1} (\HH_{\gamma,\rho})$, $p^{-1} (\HH_{\gamma,\rho}) \to
\HH_{\gamma,\rho}$, is smooth, (or see \cite{MDP1} for related arguments).
Since the fiber dimension of $p$ at $(X,Y)$ is precisely $$h^0(\sI_{X/Y}(f)) +
h^0(\sI_{X/Y}(g)) = h^0(\sI_X(f)) + h^0(\sI_X(g)) - h^0(\sI_Y(f)) -
h^0(\sI_Y(g))$$ by \cite{K3}, Thm.\! 1.16 (a), we get easily any conclusion of
the theorem if we combine with \eqref{29codim2}.
\end{proof}

\begin{remark} \label{r4.2} Let $X$ and $X'$ be two surfaces in $\proj{4}$,
  linked by a CI of type $(f,g)$. Then the arguments of the proof above show
  that we can, under the assumptions
\begin{equation} \label{van}
  H^1(\sI_X(f)) = H^1(\sI_X(g)) = 0 \ \ \ {\rm and}  \ \ \ H^1(\sI_{X'}(f)) =
  H^1(\sI_{X'}(g)) = 0 
\end{equation}
replace $\HH_{\gamma,\rho}$ and $\HH_{\gamma',\rho'}$ in Theorem~\ref{t4.1}
(i) (resp. their tangent spaces in Theorem~\ref{t4.1} (ii) ) by $\HH(d,p,\pi)$
and $\HH(d',p',\pi')$ (resp. by $H^1(\sN_X)$ and $H^1(\sN_{X'})$) and get
valid dimension formulas involving the whole Hilbert schemes (resp. their
tangent spaces). Hence assuming \eqref{van}, it follows that $X$ is
unobstructed if and only if $X'$ is unobstructed, see \cite{K3}, Prop.\! 3.12
for a generalization.
\end{remark}

\begin{example} \label{e4.3} Let $X$ be the smooth rational surface of
  $\HH(11,0,11)$ of Example~\ref{e3.10}, let $Y$ be a CI of type $(5,5)$
  containing $X$, and let $X'$ be the linked surface. Using \eqref{29} we
  deduce $\chi (\sO_{X'} (v)) = 7v^2-12v+9$ from $\chi (\sO_X(v)) =
  (11v^2-9v+2)/2$, i.e. $(X')$ belongs to $\HH(d',p',\pi') = \HH(14,8,20)$ by
  \eqref{hilbpolyX}. Moreover $\omega_{ X'} = \sI_{X/Y}(5)$ is globally
  generated (cf. the resolution of $I$ of Example~\ref{e3.10}) and the graded
  modules $M_i' \simeq \oplus H^i(\sI_{X'}(v))$ are supported at two
  consecutive degrees and satisfy
\begin{align*}
  \dim H^1 (\sI_{X'}(3)) & = 1, & \dim H^2 (\sI_{X'}(1)) & = 1, \\
  \dim H^1 (\sI_{X'}(4)) & = 3, & \dim H^2 (\sI_{X'}(2)) & = 2.
\end{align*}
From these informations we find the minimal resolution of $I' = I_{X'}$ to be
\begin{equation*}
  0 \to R(-9)^{\oplus 3} \to R(-8)^{\oplus 14} \to R(-7)^{\oplus 23} \to
  R(-6)^{\oplus 11} \oplus R(-5)^{\oplus 2} \to I' \to 0.
\end{equation*}
Combining Example~\ref{e3.10} and Remark~\ref{t3.7} we see that
$\HH_{\gamma,\rho}$ is smooth at $(X)$ and $ \dim_{(X)} \HH_{\gamma,\rho}=41$.
Thanks to Theorem~\ref{t4.1}, we get that $\HH_{\gamma',\rho'}$ is smooth at
$(X')$ and that
\begin{equation*}
  \dim_{(X')} \HH
_{\gamma',\rho'} = \dim_{(X)} \HH_{\gamma,\rho} +
  2h^0(\sI_{X/Y}(5)) -  2h^2(\sO_X(0)) = 57 .
\end{equation*}
Moreover by Remark~\ref{r4.2} or Proposition~\ref{dp3.1}, \
$\HH(d',p',\pi') \simeq \HH_{\gamma',\rho'}$  is smooth at $(X')$ and 
$\dim_{(X')} \HH(d',p',\pi') = 57$. Note that in this case we neither have
$\EExt 0{}3 (M_2, M_1) = 0$ nor $\HHom {-5}R (I, M_2) = 0$, i.e. we can not
use Corollary~\ref{p1.3} or Proposition~\ref{them} to conclude that
$\HH_{\gamma',\rho'}$ is smooth at $(X')$. But, as we have seen, the linkage
result above takes care of the smoothness and the dimension.
\end{example}
If a surface $X$ of $\proj{4}$ is contained in a CI $Y$ of type $(f,g)$, then
there is an inclusion map $I_Y \to I_X$ which induces a morphism
$l_{X/Y}^{i+1}: H^i(\sN_{X}) \to H^i(\sO_{X}(f)) \oplus H^i(\sO_{X}(g)) $ for
every $i$. We let $\beta_{X/Y}$ be the composition of $l_{X/Y}^1$ with the
natural map $H^0(\sO_{X}(f)) \oplus H^0(\sO_{X}(g)) \to H^1(\sI_{X}(f)) \oplus
H^1(\sI_{X}(g))$.
\begin{proposition} \label{vanii} Let $X$ and $X'$ be surfaces in $\proj{4}$,
  geometrically linked by a complete intersection $Y \subseteq \proj{4}$ of
  type $(f,g)$, let $(X) \in \HH_{\gamma,\rho}$ and $(X') \in
  \HH_{\gamma',\rho'}$ and suppose $\dim_{(X)} \HH_{\gamma,\rho} =
  \dim_{(X)}\HH(d,p,\pi)$. Let $c := \dim_{(X')} \HH (d',p',\pi') -
  \dim_{(X')} \HH _{\gamma',\rho'}$ and suppose $ H^1(\sI_X(f)) =
  H^1(\sI_X(g)) = 0 $ and that $l_{X/Y}^2$ is injective. Then
\begin{equation} \label{31*}
  h^1(\sI_{X'}(f)) + h^1(\sI_{X'}(g)) - h^2(\sI_{X'}(f)) -
  h^2(\sI_{X'}(g))  \le c \le  h^1(\sI_{X'}(f)) + h^1(\sI_{X'}(g))
\end{equation} 
and we have equality on the right hand side if and only if $\HH(d',p',\pi')$
is smooth at $(X')$. Furthermore, if $h^1(\sI_{X'}(v)) \cdot h^2(\sI_{X'}(v))
= 0$ for $v = f$ and $v = g$, then 
$c = h^1(\sI_{X'}(f))+h^1(\sI_{X'}(g))$.
\end{proposition}
\begin{proof} The vanishing of the obstruction group, $A^2(X \subseteq Y)$, of
  the Hilbert flag scheme $D(p(v); f,g)$ at $(X,Y)$ is equivalent to
  $\beta_{X/Y}$ being surjective and $l_{X/Y}^2$ being injective by (1.11) of
  \cite{K3}, so $A^2(X \subseteq Y)=0$ by assumption. Moreover since the
  linkage is geometric, we get $A^2(X' \subseteq Y)=0$ by Cor.\! 2.14 of
  \cite{K3}, i.e. $\beta_{X'/Y}$ is surjective, $l_{X'/Y}^2$ is injective and
  $D(p'(v); f,g)$ is smooth at $(X',Y)$. Hence \cite{K3}, Thm.\! 1.27 applies
  (onto a component $V$ satisfying $\dim V= \dim_{(X')} \HH (d',p',\pi')$) to
  get the bounds of the codimension $c$ above provided we can show that
  $\HH_{\gamma',\rho'}$ is, locally at $(X')$, an $(f,g)$-maximal subset of
  $\HH(d',p',\pi')$. By the proof of Theorem~\ref{t4.1} we see that the
  restriction of the first projection $p'$ to $p'^{-1} (\HH_{\gamma',\rho'})$,
  $p'^{-1} (\HH_{\gamma',\rho'}) \to \HH_{\gamma',\rho'}$, is smooth. It
  follows that $\HH_{\gamma',\rho'}$ is
  $(f,g)$-maximal provided we can show
  \begin{equation*} \label{dimeq} 
    \dim_{(X',Y)} p'^{-1} (\HH_{\gamma',\rho'}) =
    \dim_{(X',Y)} D(p'(v); f,g) . 
  \end{equation*} 
  Thanks to \eqref{30} it suffices to show $ \dim_{(X,Y)} p^{-1}
  (\HH_{\gamma,\rho}) = \dim_{(X,Y)} D(p(v); f,g) $ which readily follows from
  the assumptions $\dim_{(X)} \HH_{\gamma,\rho} = \dim_{(X)}\HH(d,p,\pi)$ and
  $ H^1(\sI_X(f)) = H^1(\sI_X(g)) = 0 $ because the first projection, $p :
  D(p(v); f,g) \to {\rm Hilb}^{p(v)}( \proj{4})$ and its restriction to
  $p^{-1} (\HH_{\gamma,\rho})$ are both smooth at $(X,Y)$ by
  Remark~\ref{r4.2}. Then we get the final conclusion from \cite{K3}, Cor.\!
  1.29, which states that $h^1(\sI_{X'}(v)) \cdot h^2(\sI_{X'}(v)) = 0$ for $v
  = f$ and $g$ implies that $\HH(d',p',\pi')$ is smooth at $(X')$ and we are
  done.
\end{proof}
\begin{example} \label{e4.4} Let $Z$ be the surface which is linked to the
  surface $(X') \in \HH(14,8,20)$ of Example~\ref{e4.3} via a complete
  intersection of type $(5,6)$ containing $X'$. Then $(Z)$ belongs to
  $\HH(16,15,27)$, $\omega_Z = \sI_{X'/Y}(6)$ is globally generated, and
  $M_i(Z) = \oplus H^i(\sI_Z(v))$, $i = 1,2$, are supported at two consecutive
  degrees. Moreover;
\begin{equation} \label{31} 
\begin{aligned}
    h^0 (\sI_Z(5)) & = 1, & h^1 (\sI_Z(4)) & = 2 &\text{and}&& h^1 (\sI_Z(5))
    & = 1 \\ 
    h^2 (\sO_Z(1)) & = 1, & h^2 (\sI_Z(2)) & = 3 &\text{and}&& h^2 (\sI_Z(3))
    & = 1 \ .
\end{aligned}
\end{equation} 
By Proposition~\ref{dp0.3}, we know $\chi (\sN_{X'}) =
5(2d'+\pi'-1)-d'^2+2\chi (\sO_{X'}) = 57$ and since we obviously have
$h^2(\sN_{X'}) = 0$ (from $h^2(\sO_{X'} (1)) = 0$) and we get $h^0(\sN_{X'}) =
57$ from Example~\ref{e4.3}, we conclude that $h^1(\sN_{X'}) = 0$. The
conditions of Proposition~\ref{vanii} are therefore satisfied (replacing $X$ by
$X'$ there). Hence, at $(Z)$, we get that $\HH(16,15,27)_{\gamma,\rho}$ is
smooth of codimension 1 in $\HH(16,15,27)$. Moreover $\HH(16,15,27)$ is smooth
at $(Z)$, and
\begin{equation*}
  \dim_{(Z)} \HH(16,15,27)_{\gamma,\rho} = \dim_{(X')} \HH _{\gamma',\rho'} +
  h^0(\sI_{X'/Y}(5)) + h^0(\sI_{X'/Y}(6)) - h^2(\sO_{X'}) - h^2(\sO_{X'}(1)) =
  65 .
\end{equation*}
Hence $Z$ belongs to a unique generically smooth component $V$ of
$\HH(16,15,27)$ of dimension 66, and since the generic surface $\tilde Z$ of
$V$ do not have the same cohomology as $Z$ (since $\tilde Z \notin
\HH(16,15,27)_{\gamma,\rho}$), we must get $\dim H^0(\sI_{\tilde Z}(5)) = \dim
H^1(\sI_{\tilde Z}(5)) = 0$ while elsewhere the dimension of the cohomology
groups is unchanged, i.e. it is as in \eqref{31}.
\end{example}

\section{Obstructed surfaces in  $\proj{4}$.}
In this section we explicitly prove the existence of obstructed surfaces. Our
examples are as close as they can be to the arithmetically Cohen-Macaulay case.
Indeed, in the examples, one of the Rao modules in the pair $(M_1,M_2)$
vanishes, the other is 1-dimensional. Moreover in Proposition~\ref{them} and
Remark~\ref{H1NX} we gave conditions which imply unobstructedness. Our
Example~\ref{obstructed} is minimal with respect to the mentioned conditions in
the sense that only one of the many cohomology groups, claimed in
Remark~\ref{H1NX} (i) to vanish, is non-zero. It also shows that we in
Remark~\ref{r4.2} can not skip the assumption \eqref{van} since we in
Example~\ref{obstructed} link an unobstructed surface to an obstructed surface
where 
one of the cohomology groups of \eqref{van} is non-zero. Moreover, note that
once having constructed one obstructed surface we can find infinitely many by
linking under the assumption \eqref{van}.

In the following proposition we consider a codimension 2 subscheme $X$ of
$\proj{n+2}$, containing a CI $Y$ of type $(f_1,f_2)$, in order to find
obstructed codimension 2 subschemes of $\proj{n+2}$ for $n \ge 1$. In this
situation we recall that the inclusion map $I_Y \to I_X$ induces a morphism $
H^0(\sN_{X}) \to \oplus_{i=1}^2 H^0(\sO_{X}(f_i)) $ whose composition with
$\oplus_{i=1}^2 H^0(\sO_{X}(f_i)) \to \oplus_{i=1}^2 H^1(\sI_{X}(f_i)) $ we
denote $\beta_{X/Y}$.

\begin{proposition} \label{obstreks} Let $X$ be an equidimensional locally
  Cohen-Macaulay codimension 2 subscheme of $\proj{n+2}$, and let $Y$ and
  $Y_0$ be two complete intersections containing  $X$, both of type
  $(f_1,f_2)$ such that \\[-2mm]

  \qquad {\rm i)} \qquad $ \beta_{X/Y} $ is surjective \ \ \ {\rm and} \ \ \
  $ \beta_{X/Y_0}$ \ is not surjective, \

   \qquad {\rm ii)} \qquad  $  H^n(\sI_{X}(f_i-n-3)) = 0 $  \ \ for $i=1$ and
   $i=2$. \\[2mm]
   Let $X'$ (resp. $X_0'$) be linked to $X$ by $Y$ (resp. $Y_0$). Then $X_0$ is
   obstructed. Moreover if $X$ is unobstructed, then so is $X'$.
 \end{proposition}

 \begin{proof} If $A^1(X \subseteq Y)$ is the tangent space of the Hilbert
   flag scheme $D(p(v); f_1,f_2)$ at $(X,Y)$, then it is shown in \cite{K3},
   (1.11) that there is an exact sequence $$0 \to \oplus_{i=1}^2
   H^0(\sI_{X/Y}(f_i)) \to A^1(X \subseteq Y) \to H^0(\sN_{X}) \to
   \oplus_{i=1}^2 H^1(\sI_{X}(f_i)) $$ where the rightmost map is
   $\beta_{X/Y}$. The corresponding exact sequence for $(X \subseteq Y_0)$
   together with the assumption $(i)$ show that 
   $$ \dim A^1(X \subseteq Y) < \dim A^1(X \subseteq Y_0)$$
   because it is easy to see $ h^0(\sI_{X/Y}(v)) = h^0(\sI_{X/Y_0}(v))$ for
   every $v$. We {\it claim} that $D(p(v); f_1,f_2)$ is not smooth at
   $(X,Y_0)$. Suppose the converse. Since it is shown in \cite{K3}, Thm.\!
   1.16 (a) that the fibers of the first projection $p : D(p(v); f_1,f_2) \to
   {\rm Hilb}^{p(v)}( \proj{n+2})$ are irreducible, it follows that there
   exists an irreducible component $W$ of $D(p(v); f_1,f_2)$ which contains
   both points, $(X,Y)$ and $(X,Y_0)$. Hence if $D(p(v); f_1,f_2)$ is smooth
   at $(X,Y_0)$, we get
 $$ \dim A^1(X \subseteq Y_0) = \dim W \le \dim_{(X,Y)} D(p(v); f_1,f_2) \le
 \dim A^1(X \subseteq Y),$$ i.e. a 
 contradiction. 

 Thanks to \eqref{30} we get that $D(p'(v); f_1,f_2)$ is not smooth at
 $(X_0',Y_0)$. Since $ h^1(\sI_{X_0'}(f_{i}-n-3)) = h^n(\sI_{X}(f_{3-i}-n-3))
 = 0$ for $i=1,2$, cf. \eqref{29}, and since the vanishing of $
 H^1(\sI_{X_0'}(f_{i}-n-3))$ implies that the first projection $p' : D(p'(v);
 f_1,f_2) \to {\rm Hilb}^{p'(v)}( \proj{n+2})$ is smooth at $(X_0',Y_0)$ by
 \cite{K3}, Thm.\! 1.16 (b), we conclude that $X_0'$ is obstructed. Finally,
 for the last conclusion, if we have the surjectivity of $ \beta_{X/Y} $ and
 assume the unobstructedness of $X$, we get that $D(p(v);
 f_1,f_2)$ is smooth at $(X,Y)$ by \cite{K3}, Prop.\! 3.12. Using \eqref{30}
 and \eqref{29} once more we conclude that $X'$ is unobstructed, and we are
 done.
\end{proof}

We think the surjectivity of $ \beta_{X/Y} $ may often hold, provided the
generators of $I_Y$ are among the minimal generators of $I_X$, but this is
difficult to prove. In the Buchsbaum case, however, it is easy to see the
surjectivity, as observed in \cite{BKM} for curves. Indeed even though
the statement of Proposition~\ref{obstreks} and the remark below generalizes
\cite{BKM}, Prop.\! 2.1 by far, the ideas of the proof are quite close to the
idea in Prop.\! 2.1 of \cite{BKM}.
\begin{remark} \label{obstreksrem} In this remark we consider surfaces in $
  \proj{4}$ with minimal resolution given as in \eqref{res3}.

  {\rm (i)} Using \eqref{mseq} and the spectral sequence \eqref{mspect} we get
  an exact sequence $$ \to H^0 (\sN_X) \rightarrow \HHom 0R (I_X,
  H_{\mfm}^2 (I_X)) \xrightarrow{\alpha} \EExt 0R2 (I_X,I_X) \to \ $$ where
  $ \HHom 0R (I_X, H_{\mfm}^2 (I_X)) \simeq \oplus_i H^1(\sI_X(n_{1,i}))$
  provided $ H^1(\sI_X(n_{2,i}))=0$ for any $i$. The natural map $H^0
  (\sN_X) \rightarrow \HHom 0R (I_X, H_{\mfm}^2 (I_X)) \simeq \oplus_i
  H^1(\sI_X(n_{1,i})) $, which we denote $\beta_X$, is correspondingly defined
  as $ \beta_{X/Y} $ above, but with the difference that a set of all minimal
  generators of $I_X$ is used. In particular if the generators of $I_Y$ are
  among the minimal generators of $I_X$, then the composition of $\beta_X$
  with the projection $ \oplus_i H^1(\sI_X(n_{1,i})) \to \oplus_{i=1}^2
  H^1(\sI_X(f_{i}))$ is $ \beta_{X/Y} $. It follows that if $$ \EExt 0R2
  (I_X,I_X)=0 \ \ {\rm and} \ \ H^1(\sI_X(n_{2,i}))=0 {\rm \ for \ any \ i} \
  ,$$ then $ \beta_{X/Y} $ is surjective. Note that, by \eqref{mduality} and
  \eqref{mspect} (cf. the proof of Proposition~\ref{them}), $ \EExt 0R2
  (I_X,I_X)=0 $ provided $\EExt {-5}R1(I, M_1) ={_{-5}\!\Hom_R}(I , M_2)=0,$
  i.e. provided
  $$ H^1(\sI_X (n_{2,i}-5))=0 \ {\rm and} \ H^2(\sI_X (n_{1,i}-5)) =0 \ {\rm
    for \ every} \ i. $$

  {\rm (ii)} If, however, the minimal generators $\{F_1,F_2\}$ of $I_Y$ do not
  belong to a set of minimal generators of $I_X$, say $F_i = H_i \cdot G_i$
  for some $G_i \in I_X$, $i=1,2$, then $ \beta_{X/Y} $ is easily seen to be
  non-surjective under a manageable assumption. Indeed let $g_i$ be the degree
  of the form $G_i$, let $Y_0$ be the CI with homogeneous ideal $I_{Y_0}=
  (G_1,G_2)$ and suppose the the obvious map $$h\ : \ \oplus_{i=1}^2
  H^1(\sI_X(g_{i})) \xrightarrow{(H_1,H_2)} \oplus_{i=1}^2 H^1(\sI_X(f_{i})) \
  \ { \rm is \ \ not \ \ surjective.}$$ Then $ \beta_{X/Y} $ can not be
  surjective because it factors via $h$, i.e. $ \beta_{X/Y}= h \circ
  \beta_{X/Y_0} $!
\end{remark}
\begin{example} \label{obstructed} If we link the smooth quintic scroll $Z$ of
  $\HH(5,-1,1)$ with Rao modules $H_*^1 (\sI_{Z})=0$, $H_*^2(\sI_{Z}) \simeq
  k$ and minimal resolution (cf. \cite{DES}, B.2.1),
\begin{equation} \label{scroll}
  0 \to R(-5) \to  R(-4)^{\oplus 5} \to R(-3)^{\oplus 5} \to I_Z \to 0,
\end{equation}
using a CI of type $(5,6)$ containing $Z$, then the ideal of the 
linked surface $X$ has a minimal resolution
\begin{equation*}
  0 \to R(-11)  \to R(-10)^{\oplus 5} \to R(-9)^{\oplus 10} \to R(-8)^{\oplus
    5} \oplus R(-6)  \oplus R(-5)  \to I_X \to 0
\end{equation*}
and Rao modules given by $H_*^2(\sI_{X})=0$, $h^1 (\sI_{X}(6))=1$ and $H^1
(\sI_{X}(v))=0$ for $v \ne 6$. Using \eqref{29} we see that $(X)$ belongs to
$\HH(d,p,\pi) = \HH(25,99,71)$. This surface $X$ has invariants such that
Proposition~\ref{obstreks} and Remark~\ref{obstreksrem} apply. Indeed we can
link $X$ to two different surfaces $X'$ and $X_0'$ using CI's $Y$ and $Y_0$
containing $X$, both of type $(6,8)$, generated in the following way. Let
$F_5$, resp. $F_6$, be the minimal generator of $ I_X$ of degree 5, resp. 6,
and let $G$ be a general element of $H^0(\sI_{X}(8))$. Then we take $Y$, resp.
$Y_0$, to be given by $I_Y=(F_6,G)$, resp. $I_{Y_0}=(H \cdot F_5,G)$ where
$H$is a linear form. We may check that all assumptions of
Remark~\ref{obstreksrem} are satisfied. Hence we get that $X'$ and $X_0'$
belong to a common irreducible component of $ \HH(d',p',\pi')= \HH(23,80,61)$,
that $X_0'$ is obstructed with minimal resolution
\begin{equation*}
  0 \to R(-8) \to  R(-7)^{\oplus 5} \oplus R(-8) \oplus  R(-9) \to
  R(-6)^{\oplus 6} \oplus  R(-8) \to I_{X_0'} \to 0,
\end{equation*} while $X'$ is unobstructed with minimal resolution 
\begin{equation*}
  0 \to R(-8) \to  R(-7)^{\oplus 5} \oplus  R(-9) \to
  R(-6)^{\oplus 6} \to I_{X'} \to 0.
\end{equation*}
Note that it is straightforward to find these resolutions since $X'$ and $X_0'$
are bilinked to $Z$ and we know the minimal resolution of $I_Z$, see
\cite{MIG} or the sequence \eqref{Eres} appearing later in this paper. We
observe that common direct free factors (``ghost terms'') are present in the
minimal resolution, similar to what happens for obstructed curve with ``small
Rao module'', cf. \cite{krao}. Moreover since the assumptions of
Proposition~\ref{them} are satisfied for $X'$, we also get the
unobstructedness of $X'$ from that Proposition and the dimension, $\dim_{(X')}
\HH(23,80,61)= 1 + \delta^3(-5) - \delta^2(-5) + \delta^1(-5)= 163$. However,
since the conditions of Remark~\ref{H1NX} (i) also hold, we get
$H^1(\sN_{X'})=0$ and hence it is easier to compute $\dim_{(X')}
\HH(23,80,61)$ by using Proposition~\ref{dp0.3}. We get $$\dim_{(X')}
\HH(23,80,61)= \chi (\sN_{X'})=5(2d' + \pi' - 1) - d'^2 + 2\chi (\sO_{X'})=
163.$$ Note that neither the assumptions of Proposition~\ref{them}, nor the
assumptions of Remark~\ref{H1NX} (i), are satisfied for $X_0'$. Indeed
Remark~\ref{H1NX} (i) a little extended will show $h^1(\sN_{X_0'})=1$ (i.e.
just compute the dimension using \eqref{lemm25}). The surface $X_0'$ is
reducible.
\end{example}

\begin{example} \label{obstructedsm} If we link $X_0'$ using a general CI of
  type $(9,9)$ containing $X_0'$, we get a smooth obstructed surface $S$ of
  degree 58. Indeed the assumptions of Remark~\ref{r4.2} are satisfied. So
  $S$ is obstructed, and we have used Macaulay 2 (\cite{Mac}) to verify that
  $S$ is smooth provided the CI's used in the linkages of
  Example~\ref{obstructed} are general enough under the specified
  restrictions. The surface $S$ is in the biliaison class of the Veronese
  surface in $\proj{4}$.

  Finally if we link $S$ via a general CI of type $(9,12)$ containing $S$, we
  get an obstructed surface $S'$ of degree 50 by Remark~\ref{r4.2}. We have
  used Macaulay 2 to verify that the surface is smooth. The surface $S'$ is in
  the biliaison class of the quintic elliptic scroll in $\proj{4}$. Since $S'$
  is bilinked to the surface $X_0'$ of Example~\ref{obstructed} we easily find
  the minimal resolution of $I_{S'}$ to be 
  \begin{equation*} 0 \to R(-11) \to R(-10)^{\oplus 5} \oplus R(-11) \oplus
    R(-12)^{\oplus 2} \to R(-9)^{\oplus 7} \oplus R(-11) \to I_{S'} \to 0.
\end{equation*}
Note that we again have ``ghost terms'' in the minimal resolution in degree
$c+5$ where $h^2(\sI_{S'}(c)) \ne 0$. This feature seems to be related to
obstructedness, as in the curve case, cf. \cite{krao}.
\end{example}

\section{Even liaison of codimension 2 subschemes of $\proj{n+2}$.}

In this section we prove the main even liaison theorem of this paper, which
holds for any equidimensional lCM codimension 2 subscheme $X$ of $\proj{n+2}$.
We also generalize Proposition~\ref{them} and the vanishing result for
$h^1(\sN_{X})$ of Remark~\ref{H1NX} to schemes $X$ of dimension $n>2$ and we
give an example of an obstructed 3-fold.

First we define $\delta_X^{m}(v)$. Let
\begin{equation} \label{res3cod} 0 \to \bigoplus_{i=1}^{r_{n+2}} R(-n_{n+2,i})
  \to \bigoplus_{i=1}^{r_{n+1}} R(-n_{n+1,i}) \to ... \to
  \bigoplus_{i=1}^{r_2} R(-n_{2,i}) \to \bigoplus_{i=1}^{r_1} R(-n_{1,i}) \to I
  \to 0
\end{equation}
be a  minimal resolution of $I=I_X$ and let the
invariant $ \delta^{m}(v)= \delta_X^{m}(v)$ be defined by
\begin{equation} \label{defdelta}
  \delta_X^{m}(v) = \sum_{j=1}^{n+2} \sum_{i=1}^{r_j} (-1)^{j+1}h^{m}
  (\sI_X 
  (n_{j,i}+v)) \ .
\end{equation}
Since adding common direct free factors in consecutive terms of
\eqref{res3cod} does not change $ \delta_X^{m}(v)$, the resolution of $I$
does not really need to be minimal in the definition of $\delta_X^{m}(v)$.
\begin{theorem} \label{maint4.1} Let $X$ and $X'$ be two equidimensional
  locally Cohen-Macaulay codimension 2 subschemes of $\proj{n+2}$, linked to
  each other in two steps by two complete intersections, and suppose that
  $(X)$ (resp. $(X')$) belongs to the Hilbert scheme $\HH_{\gamma,\rho}$
  (resp. $\HH_{\gamma',\rho'}$) of constant cohomology. Then
  $$ {\rm i)}  \qquad  \qquad \  \ \delta_X^{n+1}(-n-3) - \dim_{(X)}
  \HH_{\gamma,\rho} =  \delta_{X'}^{n+1}(-n-3) - \dim_{(X')}
  \HH_{\gamma',\rho'} \ . \ \qquad \qquad \ \ \  \ $$ In particular $ \ \ {\rm
    obsumext}(X):= 1+ \delta_X^{n+1}(-n-3) - \dim_{(X)} \HH_{\gamma,\rho} \ \
  $ is a biliaison invariant.
   $$ {\rm ii)} \ \ \qquad   \delta_X^{n+1}(-n-3) - \dim \EExt 0R1 (I_X,
   I_X)_\rho = \delta_{X'}^{n+1}(-n-3) - \dim \EExt 0R1 (I_{X'}, I_{X'})_{\rho
     '} \ . \ $$ In particular \ $ \ {\rm sumext}(X):=1+
   \delta_X^{n+1}(-n-3) - \dim \EExt 0R1 (I_X, I_X)_\rho \ $ \ is a biliaison
   invariant. \\[-2mm] 

   {\rm iii)} We have $ \ {\rm sumext}(X) \le {\rm obsumext}(X) $, with
   equality if and only if \ $ \HH_{\gamma,\rho}$ is smooth at $(X)$.
\end{theorem}

\begin{remark} This result is motivated by Remarks~\ref{sumex} and \ref{p3.6}.
  Indeed we were quite convinced that Theorem~\ref{maint4.1} was true before
  starting proving it. Note that the dimension formula of Remark~\ref{p3.6}
  was quite involved already for the case $n = \dim X =2$ and we expect a very
  complicated formula for $n > 2$. So Theorem~\ref{maint4.1} may be a good
  practical approach to the problem of studying $ \HH_{\gamma,\rho}$ and $
  {\rm Hilb}^{p(v)}(\proj{n+2})$ with respect to smoothness and dimension for
  $n > 2$. However, except for the other results of this paper, we have no
  better option for the use of Theorem~\ref{maint4.1} that to first compute $
  \ {\rm sumext}(X)$ and $ {\rm obsumext}(X)$ through a nice representative in
  the even liaison class, e.g. for the minimal element of the class, before we
  use it for an arbitrary element in the even liaison class.
\end{remark}
\begin{remark} \label{seminat} For the application of Theorem~\ref{maint4.1}
  there is one natural situation where $\HH_{\gamma, \rho}$ is isomorphic to $
  {\rm Hilb}^{p(v)}(\proj{n+2})$ at $(X)$, namely in the case $X$ has
  seminatural cohomology. We say a subscheme $X \subseteq \proj{n+2}$ has
  seminatural cohomology if for every $v \in \mathbb Z$, at most one of groups
  $H^0(\sI_X(v)), H^1(\sI_X(v)),..., H^{n+1}(\sI_X(v))$ are non-zero. In this
  case a generization (i.e. a deformation to more general element
  in ${\rm Hilb}^{p(v)}(\proj{n+2})$) of $X$ is forced to have the same
  cohomology as $X$ by the semicontinuity of $h^i(\sI_X(v))$, i.e.
  $\HH_{\gamma, \rho} \cong {\rm Hilb}^{p(v)}(\proj{n+2})$ as schemes at
  $(X)$.
\end{remark}
\begin{proof} Let $X$ be linked to $X_1$ by a CI \ $Y
  \subseteq \proj{n+2}$ of type $(f,g)$ and let $X_1$ be linked to $X'$ by
  some CI \ $Y' \subseteq \proj{n+2}$ of type $(f',g')$. If
  $(X_1)$ belongs to the Hilbert scheme $\HH_1:=\HH_{\gamma_1,\rho_1}$ of
  constant cohomology,
  then by Theorem~\ref{t4.1}, \\[2mm]
  $\dim_{(X_1)} \HH_{1} = \dim_{(X)} \HH_{\gamma,\rho} + h^0(\sI_{X/Y}(f)) +
  h^0(\sI_{X/Y}(g)) - h^n(\sO_X(f-n-3)) -
  h^n(\sO_X(g-n-3))$, \\[2mm]
  {\small $\dim_{(X_1)} \HH_{1} = \dim_{(X')} \HH_{\gamma',\rho'} +
    h^0(\sI_{X'/Y'}(f')) + h^0(\sI_{X'/Y'}(g')) - h^n(\sO_{X'}(f'-n-3)) -
    h^n(\sO_{X'}(g'-n-3))$.}\\[3mm]
  Let $h=f'+g'-f-g$. Using \eqref{29codim2} twice we get $ h^0(\sI_{X'/Y'}(v))=
  h^0(\sI_{X/Y}(v-h))$. Hence {\small 
    \begin{equation} \label{hetta} \dim_{(X')} \HH_{\gamma',\rho'}= \dim_{(X)}
      \HH_{\gamma,\rho} + h^0(\sI_{X/Y}(f)) + h^0(\sI_{X/Y}(g))-
      h^0(\sI_{X/Y}(f'-h)) + h^0(\sI_{X/Y}(g'-h))+ \eta 
    \end{equation} } where
  $\eta$ is defined by {\small
    \begin{equation} \label{etta} \eta := h^n(\sO_{X'}(f'-n-3))+
      h^n(\sO_{X'}(g'-n-3))- h^n(\sO_{X}(f-n-3))-h^n(\sO_{X}(g-n-3)).
    \end{equation}
  } Next we need to find a free resolution of $I'=I_{X'}$ in terms of the
  minimal resolution of $I=I_{X}$ in \eqref{res3cod}. If we define $E$ by the
  exact sequence
  \begin{equation} \label{E} 0 \to \oplus_{i=1}^{r_{n+2}} R(-n_{n+2,i}) \to
    ... \to \oplus_{i=1}^{r_{3}} R(-n_{3,i}) \to \oplus_{i=1}^{r_2}
    R(-n_{2,i}) \to E \to 0,
\end{equation} 
we may put \eqref{res3cod} in the form $ 0 \to E \to \oplus_{i=1}^{r_1}
R(-n_{1,i}) \to I \to 0.$ Then it is well known that there is an exact
sequence 
\begin{equation} \label{Eres} 0 \to E(-h) \oplus R(-f-h) \oplus R(-g-h) \to
  \oplus_{i=1}^{r_1} R(-n_{1,i}-h) \oplus R(-f') \oplus R(-g') \to I' \to
  0
\end{equation} which combined with \eqref{E} yields a free resolution of
$I'$ (see \cite{MIG}). 

We will use this resolution of $I'$ and \eqref{res3cod} to see the connection
between $\delta_X^{n+1}(-n-3)$ and $\delta_{X'}^{n+1}(-n-3)$. First we need to
compute $\beta$ defined by
$$ \beta:= 
\sum_{j=1}^{n+2} \sum_{i=1}^{r_j} (-1)^{j+1}\alpha(n_{j,i}-n-3) \ \ {\rm 
  where } \ \ \alpha(v):= h^n(\sO_{X'}(v+h)) - h^n(\sO_{X}(v)) \ . $$ We {\it
  claim} that
\begin{equation} \label{betta} \beta = h^0(\sI_{X}(f)) + h^0(\sI_{X}(g))-
  h^0(\sI_{X}(f'-h))- h^0(\sI_{X}(g'-h)) + h^0(\sI_{X}(-h)).
\end{equation}  Indeed by
\eqref{29codim2}, \ \ $\alpha(v)= h^0(\sI_{X_1/Y'}(f'+g'-n-3-v-h)) -
h^0(\sI_{X_1/Y}(f+g-n-3-v))$.
\\[1mm]
Moreover since $0 \to \sI_{Y'} \to \sI_{X_1} \to \sI_{X_1/Y'} \to 0$ and $0
\to \sI_{Y} \to \sI_{X_1} \to \sI_{X_1/Y} \to 0$ are exact, we get
\begin{equation} \label{alfa}
\alpha(v)= h^0(\sI_{Y}(f+g-n-3-v)) - h^0(\sI_{Y'}(f+g-n-3-v)).
\end{equation}
Let $r(v):= \dim R_{(-n-3+v)}$. Combining with the minimal resolutions of
$I_Y$ and $I_Y'$, we get
$$\alpha(v):= r(f-v)+ r(g-v) - r(-v)- r(f'-h-v)- r(g'-h-v)+ r(-h-v).$$
Then we get the claim since \eqref{res3cod} implies $ h^0(\sI_{X}(v))=
\sum_{j=1}^{n+2} \sum_{i=1}^{r_j} (-1)^{j+1} r(v-n_{j,i}+n+3)$ for any $v$ and
since $h^0(\sI_{X}(0))=0$.

Using the resolution of $I'$ deduced from \eqref{Eres} and the definition
\eqref{defdelta} we get $$\delta_{X'}^{n+1}(-n-3)= \sum_{j=1}^{n+2}
\sum_{i=1}^{r_j} (-1)^{j+1} h^n(\sO_{X'}(n_{j,i}+h-n-3))+ \epsilon $$ where
$\epsilon$ is defined by {\small $$\epsilon := h^n(\sO_{X'}(f'-n-3))+
  h^n(\sO_{X'}(g'-n-3))-
  h^n(\sO_{X'}(f+h-n-3))-h^n(\sO_{X'}(g+h-n-3)).$$}Comparing $\epsilon$ with
$\eta$ in $\eqref{etta}$ and recalling the definition of $ \alpha$, we have
$\epsilon = \eta -\alpha(f-n-3) -\alpha(g-n-3)$. Moreover the definition of $
\alpha$, the proven claim and \eqref{defdelta} lead to
$\delta_{X'}^{n+1}(-n-3)= \delta_{X}^{n+1}(-n-3)+ \beta +\epsilon .$ Combining
we get
$$\delta_{X'}^{n+1}(-n-3)= \delta_{X}^{n+1}(-n-3)+ \beta +\eta -\alpha(f-n-3)
-\alpha(g-n-3) .$$ Comparing with \eqref{hetta} we get (i) of the Theorem
provided we can show that $$ h^0(\sI_{X/Y}(f)) + h^0(\sI_{X/Y}(g))-
h^0(\sI_{X/Y}(f'-h)) - h^0(\sI_{X/Y}(g'-h))= \beta -\alpha(f-n-3)
-\alpha(g-n-3) .$$
Suppose $h \ge 0$. Looking to \eqref{betta}, it suffices to show
 $$ -h^0(\sI_{Y}(f)) - h^0(\sI_{Y}(g))+
 h^0(\sI_{Y}(f'-h)) + h^0(\sI_{Y}(g'-h))= -\alpha(f-n-3) -\alpha(g-n-3) .$$
 Thanks to \eqref{alfa} it remains to show $ h^0(\sI_{Y}(f'-h)) +
 h^0(\sI_{Y}(g'-h)) = h^0(\sI_{Y'}(f)) + h^0(\sI_{Y'}(g))$. Using the minimal
 resolutions of $I_Y$ and $I_{Y'}$ and that $h=f'+g'-f-g \ge 0$, we easily
 show that both sides of the last equation is equal to $\dim R_{(f-f')}+\dim
 R_{(f-g')}+\dim R_{(g-f')}+\dim R_{(g-g')}$ and we get what we want, i.e. 
\begin{equation} \label{mainthi}
 \delta_X^{n+1}(-n-3) - \dim_{(X)} \HH_{\gamma,\rho} = \delta_{X'}^{n+1}(-n-3)
 - \dim_{(X')} \HH_{\gamma',\rho'}
\end{equation} provided $h \ge 0$. Suppose $h<0$. Then we can start with $X'$
and link in two steps back to $X$, i.e. we get an even liaison with
$h' = f+g-f'-g' \ge 0$ in which case we know that \eqref{mainthi} holds. Hence
\eqref{mainthi}  is proved in general.

To show (ii) of the Theorem we only need to remark that, due to
Theorem~\ref{t4.1}, \eqref{hetta} holds if we replace $\dim_{(X)} \HH_{\gamma,
  \rho}$ and $\dim_{(X')} \HH_{\gamma',\rho'}$ by the dimension of their
tangent spaces $\EExt 0R1 (I_X, I_X)_\rho$ and $\EExt 0R1
(I_{X'},I_{X'})_{\rho'}$ respectively. With the proof of Theorem~\ref{maint4.1}
(i) above, we therefore get \eqref{mainthi} with the mentioned replacements,
i.e. we get Theorem~\ref{maint4.1} (ii).

Finally Theorem~\ref{maint4.1} (iii) follows by combining (i) and (ii) since
e.g. the smoothness of $\HH_{\gamma, \rho}$ at $(X)$ is equivalent to
$\dim_{(X)} \HH_{\gamma, \rho} = \dim \EExt 0R1 (I_X, I_X)_\rho$.
\end{proof}

\begin{corollary}\label{thmeulerICMhigher}
  Let $X$ be an equidimensional lCM codimension 2 subschemes of $\proj{n+2}$,
  and suppose $(X)$ be a generic point of a generically smooth component $V$
  of $ {\rm Hilb}^{p(v)}(\proj{n+2})$. Then $ \ {\rm sumext}(X) = {\rm
    obsumext}(X) $ and 
\begin{equation*}
  \dim V =  1 + \delta_X^{n+1}(-n-3)- {\rm sumext}(X).
\end{equation*}
\end{corollary}
\begin{proof}
  Arguing as the last part of the proof of Theorem~\ref{thmeulerICM}, we get
  that $\HH_{\gamma, \rho}$ is isomorphic to $ {\rm Hilb}^{p(v)}(\proj{n+2})$
  at $(X)$. Hence $\HH_{\gamma, \rho}$ is smooth at $(X)$. Then we
  conclude by Theorem~\ref{maint4.1}.
\end{proof}

\begin{corollary} \label{corp1.3} Let $X$ be a surface in $\proj{4}$. If the
  local deformation functors $Def(M_i)$ of $M_i$ are formally smooth (for
  instance if $\EExt 0R2 (M_i, M_i) = 0$) for $i = 1,2$, and if $ \EExt 0R3
  (M_2, M_1) = 0 $, then
\begin{equation*}
 \ {\rm sumext}(X) = {\rm obsumext}(X).
\end{equation*}
\end{corollary}
\begin{proof}
  By Corollary~\ref{p1.3} we get that $\HH_{\gamma, \rho}$ is smooth at
  $(X)$ and we conclude by Theorem~\ref{maint4.1} {\rm (iii)}.
\end{proof}

\begin{corollary} \label{themACM} Let $X$ be an arithmetically Cohen-Macaulay
  codimension 2 subschemes of $\proj{n+2}$. Then $ {\rm
    sumext(X)}= {\rm obsumext}(X) = 0$. Moreover,  \\[-3mm]

  {\rm (i)} if $n > 0$, then $X$ is unobstructed and $$ \dim_{(X)} {\rm
    Hilb}^{p(v)}( \proj{n+2}) = 1 + \delta_X^{n+1}(-n-3) =1-\delta_X^0(0) =
  \chi(\sN_X) + (-1)^n \delta_X^0(-n-3),$$
  \\[-8mm]

  {\rm (ii)} if $n = 0$, then $\HH_{\gamma}$ is smooth at $(X)$ and $$
  \dim_{(X)}\HH_{\gamma} = 1 + \delta_X^{1}(-3)
  =1-\delta_X^0(0)=h^0(\sN_X)+\delta_X^0(-3).$$

\end{corollary}
\begin{proof}
  By Gaeta's theorem (\cite{Ga}, \cite{Ga2}, cf.\! \cite{Ap2}) $X$ is in the
  liaison class of a complete intersection $Y$. Suppose $n > 0$. Then $
  \HH_{\gamma,\rho} \cong \HH_{\gamma} \cong {\rm Hilb}^{p(v)}( \proj{n+2})$
  at $(X)$ by \cite{El} or \cite{K79}, Rem.\! 3.7, (cf. \cite{W1}, Thm.\! 2.1).
  Thanks to Theorem~\ref{maint4.1} it suffices to show that $ {\rm
    sumext(Y)}=0$, or equivalently that $ \dim \EExt 0R1 (I_Y, I_Y)_\rho = 1 +
  \delta_Y^{n+1}(-n-3)$. By definition, cf.\! \eqref{tansp}, and \eqref{mseq},
  $ \EExt 0R1 (I_Y, I_Y)_\rho = \EExt 0R1 (I_Y, I_Y) = h^0(\sN_Y)$ and it is
  trivial to show $ h^0(\sN_Y) = 1 + \delta_Y^{n+1}(-n-3)$ by using duality
  and the minimal resolution of $I_Y$.

  Moreover note that for {\it any} equidimensional lCM codimension 2
  subschemes $X$ of $\proj{n+2}$, we easily show
  \begin{equation} \label{1delta} \sum_{i=1}^{n+1} \eext 0Ri (I_X,I_X) = 1 -
    \delta_X^0(0) = \chi(\sN_X) + (-1)^n \delta_X^0(-n-3).
  \end{equation} as in
  Proposition~\ref{dp0.3} (see the first sentence of the proof for the left
  equality and second and third sentence of  the proof for the right
  equality).  
  Hence if $X$ is arithmetically Cohen-Macaulay we get $ \eext 0Ri (I_X,I_X) =
  0$ for $i \ge 2$ and we are done in the case $n > 0$. The case $n=0$ is
  similar and easier.
\end{proof}

\begin{remark}
  Corollary~\ref{themACM} coincides with \cite{El} if $n > 0$, and with
  \cite{Got} and \cite{KM}, Rem.\! 4.6 if $n=0$.
\end{remark}

\begin{example} \label{e5.3} Let $X$ be the smooth rational surface of
  $\HH(11,0,11)$ of Example~\ref{e3.10}. Note that $X$ has seminatural
  cohomology and hence we have $\HH_{\gamma,\rho} \cong \HH(d,p,\pi)$ at $(X)$
  by Remark~\ref{seminat}. Moreover $I = I_X$ admits a minimal resolution
\begin{equation} \label{sumextex1}
  0 \to R(-9) \to R(-8)^{\oplus 3} \oplus R(-7)^{\oplus 3} \to R(-7)^{\oplus
    2} \oplus
  R(-6)^{\oplus 12} \to R(-5)^{\oplus 10} \to I \to 0 .
\end{equation}
By Example~\ref{e3.10} we conclude that $\HH_{\gamma,\rho} \cong \HH(d,p,\pi)$
is smooth at $(X)$ and $\dim_{(X)} \HH(d,p,\pi) = 41$. However, since $X$ is
rational we obviously get $ 1+ \delta_X^{3} (-5)=1$ from \eqref{sumextex1}. By
Theorem~\ref{maint4.1} we find $ {\rm sumext}(X)= {\rm obsumext}(X) = -40$.
Now we link twice to get $X'$, first using a CI of type $(5,5)$, then a CI of
type $(5,6)$, both times using a common hypersurface of degree 5. Looking to
\eqref{Eres} we find a free resolution of $I'=I_{X'}$ of the form {\small
\begin{equation} \label{sumextex11} 0 \to R(-10) \to R(-9)^{\oplus 3} \oplus
  R(-8)^{\oplus 3} \to R(-8)^{\oplus 2} \oplus R(-7)^{\oplus 12} \oplus R(-6)
  \to R(-6)^{\oplus 10} \oplus R(-5) \to I' \to 0 .
\end{equation}
} By \eqref{29codim2}, $h^2(\sO_{X'})= 15$ and $h^2(\sO_{X'}(1))= 1$ and we get
$1+ \delta_{X'}^{3}(-5)=25$. It follows from Theorem~\ref{maint4.1} and
Proposition~\ref{dp3.1} that $\HH_{\gamma',\rho'} \cong \HH(d',p',\pi')$ is
smooth at $(X')$ of dimension $1+ \delta_{X'}^{3}(-5)- {\rm sumext}(X)= 65$.
Compare with Examples~\ref{e4.3} and \ref{e4.4}.
\end{example}
Before considering examples of 3-folds, we want to generalize some of
the results of section 4. For recent papers on the Hilbert scheme of
3-folds, see \cite{BF} and its references. See also \cite{DPo} for a long list
of examples of 3-folds of non general type.

\begin{proposition} \label{themhigher} Let $X$ be an equidimensional lCM
  codimension 2 subschemes of $\proj{n+2}$, let $M_i = H_*^i(\sI_X)$ for i =
  1,...,n and $I=I_X$ and suppose $$ {_{0}\!\Hom_R}(I, M_1)= 0 \ \ {\rm and}\
  \ \EExt {-n-3}R{n-j}(I, M_j)=0 \ {\rm for \ every \ } j , \ 1 \le j \le n.$$
  Then $ \EExt 0R2 (I,I)=0$, $X$ is unobstructed and $$ \dim_{(X)} {\rm
    Hilb}^{p(v)}(\proj{n+2}) = \eext 0R1 (I,I).$$ E.g. let $ \dim X = 3$. Then
  $X$ is unobstructed and \ $ \dim_{(X)} {\rm Hilb}^{p(v)}(\proj{5}) = \eext
  0R1 (I,I)$ \ if, for every $i$, 
\begin{equation} \label{ext20}  \qquad H^1(\sI_X(n_{1,i}))= H^3(\sI_X
  (n_{1,i}-6)) =0, \ H^2(\sI_X (n_{2,i}-6)) =0, \ {\rm and} \ H^1(\sI_X
  (n_{3,i}-6))=0. 
\end{equation} If in addition 
\begin{equation} \label{extj0} \ H^2(\sI_X (n_{1,i}-6)) =0, \ H^1(\sI_X
  (n_{2,i}-6))=0 \ \ {\rm and} \ \ H^1(\sI_X (n_{1,i}-6))=0,
\end{equation} then $ \quad
\dim_{(X)} {\rm Hilb}^{p(v)}(\proj{5}) = 1 - \delta_X^0(0)= \chi(\sN_X)
-\delta_X^0(-6).$ 
\end{proposition}

\begin{proof} Thanks to \cite{K79}, Rem.\! 3.7 (cf. \cite{W1}, Thm.\! 2.1), the
  Hilbert scheme $\HH_{\gamma}$ of constant postulation is isomorphic to $
  {\rm Hilb}^{p(v)}(\proj{n+2})$ at $(X)$ provided $ {_{0}\!\Hom_R}(I,M_1) =
  0$. By \eqref{mduality} we get $ \EExt 0R2 (I,I)=0$ provided
  ${_{-n-3}\!\Ext_{\mathfrak m}^{n+1}}(I,I)=0$. By \eqref{mspect} and $M_j
  \cong H_{\mathfrak m}^{j+1}(I)$ we deduce the vanishing of the latter from
  the assumptions of the proposition. It follows that $\HH_{\gamma}$ is smooth
  at $(X)$ of dimension $ \eext 0R1 (I,I)$. 

  Suppose $n = 3$. By the definition of $ \EExt vR{\bullet} (I,-)$ and
  \eqref{res3cod} we easily prove the vanishing of all $ \EExt {}R{\bullet}
  (I,-)$-groups of the first part of the proposition from the explicit
  vanishings in \eqref{ext20}. Moreover due \eqref{1delta}, to get the final
  formula it suffices to show $ \EExt 0Rj (I,I)=0$ for $j=3,4$. By
  \eqref{mduality} we must prove ${_{-n-3}\!\Ext_{\mathfrak m}^{n-j}}(I,I)=0$
  for $j=0,1$. This is shown in exactly the same way as we did for
  ${_{-n-3}\!\Ext_{\mathfrak m}^{n+1}}(I,I)=0$, i.e. by using \eqref{mspect}
  and \eqref{res3cod} and we are done.
\end{proof}
\begin{remark} \label{HiNX} {\rm (i)} We can also generalize Remark~\ref{H1NX}
  to equidimensional lCM codimension 2 subschemes $X \subseteq \proj{n+2}$ of
  higher dimension. Indeed using \eqref{mseq}, \eqref{mspect} and
  \eqref{mduality}, see the proof above, we get $H^1(\sN_X)=0$ provided $
  \EExt 0{\mathfrak m}3(I,I)=0$ and ${_{-n-3}\!\Ext_{\mathfrak
      m}^{n+1}}(I,I)=0$, e.g. provided
  $$ \EExt {0}R{j}(I, M_{2-j})= 0 \ {\rm for} \ 0 \le j \le 1 \ \ {\rm and} \
  \ \EExt {-n-3}R{n-j}(I, M_j)=0 \ {\rm for \ every} \ 1 \le j \le n.$$
Similarly  $H^2(\sN_X)=0$ provided
  $ \EExt 0{\mathfrak m}4(I,I)=0$ and ${_{-n-3}\!\Ext_{\mathfrak
      m}^{n}}(I,I)=0$, e.g. provided
  $$ \EExt {0}R{j}(I, M_{3-j})= 0 \ \ {\rm for} \ 0 \le j \le 2 \ \ {\rm and} \
  \ \EExt {-n-3}R{n-j}(I, M_{j-1})=0 \ \ {\rm for} \ 2 \le j \le n.$$ We can
  in this way easily get a vanishing criteria for $H^q(\sN_X)=0$ for every $q
  \ge 1$. \\[1mm]
  {\rm (ii)} Suppose for instance $n = \dim X = 3$. Then $H^1(\sN_X)=0$ if,
  for every $i$, {\small $$ H^1(\sI_X (n_{2,i}))= H^2(\sI_X (n_{2,i}-6)) =0, \
    H^2(\sI_X(n_{1,i}))= H^3(\sI_X (n_{1,i}-6)) =0 \ {\rm and} \ H^1(\sI_X
    (n_{3,i}-6))=0. $$} Moreover $H^2(\sN_X)=0$ if, for every $i$, {\small $$
    H^1(\sI_X (n_{3,i}))= 0, \ H^2(\sI_X (n_{2,i}))= H^1(\sI_X (n_{2,i}-6))=0 \
    \ {\rm and} \ \ H^3(\sI_X(n_{1,i}))= H^2(\sI_X (n_{1,i}-6)) =0. $$}
\end{remark}
As in the surface case, if some of the assumptions of
Proposition~\ref{themhigher} or Remark~\ref{HiNX} are not satisfied, we can
find examples of obstructed 3-folds (e.g. $X_0'$ in the example below). Note
that all assumptions of Proposition~\ref{themhigher}
and Remark~\ref{HiNX} (ii) are satisfied for $X_0'$, except $ H^3(\sI_{X_0'}
(n_{1,i}-6)) =0 $ for one $i$.
\begin{example} \label{3obstructed} We start with the smooth 3-fold $Z
  \subseteq \proj{}:=\proj{5}$ of \cite {ok} of degree 7 with
  $\Omega$-resolution $$0 \to \sO_{ \proj{}}^{\oplus 4} \to
  \Omega_{\proj{}}(2) \to \sI_Z(4) \to 0,$$ where $ \Omega_{\proj{}}$ is the
  kernel of the map $\sO_{\proj{}}(-1)^6 \to \sO_{\proj{}}$ induced by the
  multiplication with $(X_0,..,X_5)$. Note that $ h^1(\sI_Z (2))=1$. If we
  link $Z$, first using a CI of type $(4,4)$ to get a 3-fold $Z'$, then a CI
  of type $(6,7)$ to link $Z'$ to $X$, then $X$ is a 3-fold with properties
  such that Proposition~\ref{obstreks} applies. Indeed the ideas of
  Remark~\ref{obstreksrem} also apply except for how we proved $ \EExt 0R2
  (I,I)=0$. By the proof of Proposition~\ref{themhigher}, however, we have $
  \EExt 0R2 (I,I)=0$ for 3-folds provided $ H^3(\sI_X (n_{1,i}-6)) = \
  H^2(\sI_X (n_{2,i}-6)) = H^1(\sI_X (n_{3,i}-6))=0 $ for all $i$. To see that
  all these $ H^i(\sI_X(j))$-groups vanish, we first find the minimal
  resolution of $I_{Z'}$. Combining the exact sequence $0 \to \sO_{\proj{}}
  \to \sO_{\proj{}}(1)^6 \to \Omega_{\proj{}}^{\vee} \to 0 $ with the mapping
  cone construction for how we get the resolution of $I_{Z'}$ from the
  resolution of $I_{Z}$, we find the minimal resolution
\begin{equation*}
  0 \to R(-6) \to  R(-5)^{\oplus 6}  \to
  R(-4)^{\oplus 6} \to I_{Z'} \to 0.
\end{equation*} 
Hence $H_*^1 (\sI_{Z'})=0$, $H_*^2(\sI_{Z'})=0$ and we get $H_*^3
(\sI_{X})=0$, $H_*^2(\sI_{X})=0 $ and $H_*^1(\sI_{X}) \simeq H^1(\sI_{X}(7))
\simeq k $, cf. \eqref{29}. Now since the Koszul resolution induced by the
regular sequence $ \{ X_0,..,X_5 \} $ implies that $$ 0 \to \sO_{ \proj{}}(-6)
\to \sO_{ \proj{}}(-5)^{\oplus 6} \to \sO_{ \proj{}}(-4)^{\oplus 15} \to \sO_{
  \proj{}}(-3)^{\oplus 20} \to \sO_{ \proj{}}(-2)^{\oplus 15} \to
\Omega_{\proj{}} \to 0 $$ is exact, we can use the mapping cone construction
to find the following $\Omega$-resolution, $$0 \to \sO_{ \proj{}}(-9)^{\oplus
  6} \to \Omega_{\proj{}}(-7) \oplus \sO_{ \proj{}}(-7) \oplus \sO_{
  \proj{}}(-6) \to \sI_X \to 0$$ of $ \sI_X$, leading to the minimal resolution
 $$0 \to R(-13)
 \to R(-12)^{\oplus 6} \to R(-11)^{\oplus 15} \to ... \to I_X \to 0. $$ It
 follows that all $n_{3,i}=11$ in the minimal resolution of $I_X$ and hence we
 see that $ \EExt 0R2 (I,I)=0$.

 Then we proceed exactly as in Example~\ref{obstructed}. Indeed we link $X$ to
 two different 3-folds $X'$ and $X_0'$ using CI's $Y$ and $Y_0$ containing
 $X$, both of type $(7,9)$, as follows. Let $F_6$, resp. $F_7$, be the minimal
 generator of $ I_X$ of degree 6, resp. 7, and let $G$ be a general element of
 $H^0(\sI_{X}(9))$. Then we take $Y$, resp. $Y_0$, to be given by
 $I_Y=(F_7,G)$, resp. $I_{Y_0}=(H \cdot F_6,G)$ where $H$is a linear form. We
 may check that all assumptions of Proposition~\ref{obstreks} are satisfied.
 Hence we get that $X'$ and $X_0'$ belong to a common irreducible component of
 $ {\rm Hilb}^{p(v)}(\proj{5})$, that $X_0'$ is obstructed with minimal
 resolution 
\begin{equation*}
  0 \to R(-9) \to  R(-8)^{\oplus 6} \oplus R(-9) \oplus  R(-10) \to
  R(-7)^{\oplus 7} \oplus  R(-9) \to I_{X_0'} \to 0,
\end{equation*} 
cf.\! \eqref{Eres}, while $X'$ is unobstructed with minimal resolution
\begin{equation*}
  0 \to R(-9) \to  R(-8)^{\oplus 6} \oplus  R(-10) \to
  R(-7)^{\oplus 7} \to I_{X'} \to 0.
\end{equation*}
Again we have ``ghost terms'' in the minimal resolution of $ I_{X_0'}$. From
the resolution we find $X_0'$ to be of degree 30 and with Hilbert
polynomial $$p(v)=5v^3- \frac{67}{2}v^2+ \frac{247}{2}v-153.$$ The 3-fold
$X_0'$ is reducible.
Moreover since the assumptions of Proposition~\ref{themhigher} are satisfied
for $X'$, we also get the unobstructedness of $X'$ from that Proposition and
the dimension, $\dim_{(X')} {\rm Hilb}^{p(v)}(\proj{5})= 1 -
\delta_{X'}^0(0) = 327$.
Note that the assumptions of
Proposition~\ref{themhigher} 
are not satisfied for $X_0'$, due to the existence of a minimal
generator of degree 9 of  $ I_{X_0'}$ and the fact $h^3(\sI_{X_0'}(3))=1$.

Finally since Remark~\ref{r4.2} generalizes to 3-folds by \cite{K3}, Prop.\!
3.12, one may by linkage obtain infinitely many obstructed 3-folds in the
liaison class of $X_0'$.
\end{example}
  Finally we recall the Hilbert polynomials of $\sO_X$ and $\sN_X$ for an
  equidimensional lCM $3$-fold of degree $d$ and sectional genus $\pi$. If $S$
  is a general hyperplane section, we have an exact sequence $$ 0 \to
  \sO_X(v-1) \to \sO_X(v) \to \sO_S(v) \to 0, $$ and we easily deduce
\begin{equation} \label{hilbpolyX3} p(v):= \chi (\sO_X(v)) = \frac{1}{6} dv^3 +
  \frac{1}{2}(d+1 - \pi)v^2 +( \chi (\sO_S)+ \frac d3 + \frac {1-\pi}{2})v +
  \chi (\sO_X)
\end{equation}
from \eqref{hilbpolyX}. Moreover
\begin{proposition} \label{dp0.3higher} Let $X$ be an equidimensional lCM
  3-fold in $\proj{5}$ of degree $d$ and
  sectional genus $\pi$ and let $S$
be a general hyperplane section. Then 
\begin{equation*} \label{hilbpolyN3}
  \chi (\sN_X(v)) = \frac{1}{3} dv^3 + 3dv^2+ (2 \chi (\sO_S)+ 5(\pi-1)+ \frac
  {38}{3} d-d^2)v + (6 \chi (\sO_S)+ 15(\pi-1)+ 20d-3d^2).
\end{equation*}
\end{proposition}

\begin{proof} Since we have no reference for this formula in this generality
  we sketch a proof. Indeed we claim that
  \begin{equation} \label{hilbpolyNX3} \chi (\sN_X(v)) = \chi (\sO_X(v)) -
    \chi (\sO_X(-v-6)) - d^2(v+3).
\end{equation}
Note that, using \eqref{hilbpolyNX3}, we get Proposition~\ref{dp0.3higher} by
combining with \eqref{hilbpolyX3}. To show \eqref{hilbpolyNX3}, we follow the
proof of Proposition~\ref{dp0.3}. In addition to the formulas in \eqref{nijX}
(where we only replace $\sum_{j=1}^4$ by $\sum_{j=1}^5$) we get
$\sum_{j=1}^5(-1)^{j-1} \sum_i n_{j,i}^3=6(1-\pi-2d).$ Then we proceed as in
\eqref{lem23}. We get $ \delta^0(v) = -\chi(\sI_{X}(-v-6))-
\chi(\sO_X(v))+(3+v)d^2$ for $v>>0$ and then the claim.
\end{proof}

\begin{example} \label{e5.4} Let $X$ be the smooth Calabi-Yau 3-fold of
  \cite{DPo}, sect.\! 6, with invariants $d=17$, $\pi= 32$, $\chi(\sO_X)=0$
  and $\chi(\sO_S)=24$, and deficiency
  modules $M_1=0$, $M_2=0$ and $ H^3(\sI_X(v))=0$ except when
   \begin{align*} h^3(\sI_X(1)) = 4, \quad h^3 (\sI_X(2)) = 2.
\end{align*}
Following \cite{DPo} we find that $I = I_X$ has the following minimal
resolution
\begin{equation} \label{sumextex31} 0 \to R(-8)^{\oplus 2} \to R(-7)^{\oplus 8}
  \to  R(-6)^{\oplus 5} \oplus R(-5)^{\oplus 2} \to I \to 0 .
\end{equation}
All assumptions of Proposition~\ref{themhigher} are satisfied and we get that
$ {\rm Hilb}^{p(v)}(\proj{5})$ is smooth at $(X)$ of dimension $$
\dim_{(X)} {\rm Hilb}^{p(v)}(\proj{5}) = 1 - \delta_X^0(0)= 82.$$
Let us compute ${\rm obsumext}(X)$. Note that $X$ has seminatural cohomology
and hence we have $\HH_{\gamma,\rho} \cong {\rm Hilb}^{p(v)}(\proj{5})$ at
$(X)$ by Remark~\ref{seminat}. Since we have $h^3(\sO_{X})= 1$ and
$h^3(\sO_{X}(-1))= 24$, it follows that ${\rm obsumext}(X)= 1+
\delta_X^{3}(-6)-82=-28$ by Theorem~\ref{maint4.1}. Now we link twice to get
$X'$, first using a CI of type $(5,6)$, then a CI of type $(5,5)$, both times
using a common hypersurface of degree 5. This is possible, cf. \cite{DPo}.
Thanks to \eqref{Eres} we find a free resolution of $I'=I_{X'}$ of the form
  \begin{equation} \label{sumextex311} 0 \to R(-7)^{\oplus 2} \to
    R(-6)^{\oplus 8} \to R(-5)^{\oplus 6} \oplus R(-4) \to I' \to 0 .
\end{equation}
By \eqref{29codim2} $h^1(\sO_{X'}(-2))= 19$ and $h^1(\sO_{X'}(-1))= 0$ and we
get $1+ \delta_{X'}^{3}(-6)=20$. It follows from Theorem~\ref{maint4.1} and
Proposition~\ref{dp3.1} that $\HH_{\gamma',\rho'} \cong {\rm
  Hilb}^{p'(v)}(\proj{5})$ is smooth at $(X')$ of dimension $1+
\delta_{X'}^{3}(-5)- {\rm sumext}(X)= 48$. We can also use
Proposition~\ref{themhigher} and check that $ 1 - \delta_{X'}^0(0)= 48.$
\end{example}

\end{document}